\documentclass[a4paper,12pt,dvipdfmx]{amsart}
\usepackage{amsmath}
\usepackage{amssymb}
\usepackage{enumerate}
\usepackage{amscd}
\usepackage{graphics}
\usepackage{latexsym}
\usepackage{verbatim} 
\usepackage[all]{xy}

\theoremstyle{plain}
\newtheorem{thm}{{\bf Theorem}}[section]
\newtheorem{cor}[thm]{{\bf  Corollary}}
\newtheorem{prop}[thm]{{\bf Proposition}}
\newtheorem{lemma}[thm]{{\bf Lemma}}
\newtheorem{fact}[thm]{{\bf Fact}}
\newtheorem{claim}[thm]{{\bf Claim}}

\theoremstyle{definition}
\newtheorem{define}[thm]{{\bf Definition}}
\newtheorem{question}[thm]{{\bf Question}}
\newtheorem{remark}[thm]{{\bf Remark}}

\newcommand{\cf}{\mathord{\mathrm{cf}}}

\newcommand{\dom}{\mathord{\mathrm{dom}}}
\newcommand{\size}[1]{\left\vert {#1} \right\vert}
\newcommand{\seq}[1]{\langle {#1} \rangle}

\newcommand{\ot}{\mathord{\mathrm{ot}}}

\newcommand{\col}{\mathord{\mathrm{Col}}}

\newcommand{\supp}{\mathrm{supp}}

\newcommand{\FRP}{\mathsf{FRP}}

\newcommand{\ZFC}{\mathsf{ZFC}}

\newcommand{\List}{\mathrm{List}}
\newcommand{\Refl}{\mathsf{RP}}

\newcommand{\Cof}{\mathrm{Cof}}
\newcommand{\coll}{\mathrm{Coll}}

\newcommand{\ka}{\kappa}
\newcommand{\la}{\lambda}
\newcommand{\om}{\omega}

\newcommand{\force}{{\Vdash}}

\newcommand{\bbP}{\mathbb{P}}
\newcommand{\bbQ}{\mathbb{Q}}

\newcommand{\bbS}{\mathbb{S}}

\newcommand{\calE}{\mathcal{E}}

\newcommand{\calF}{\mathcal{F}}

\newcommand{\calH}{\mathcal{H}}

\newcommand{\calV}{\mathcal{V}}

\newcommand{\relE}{\mathrel\mathcal E}

\newcommand{\restr}{\restriction}
\newcommand{\rest}{\restriction}

\title{The list-chromatic number and the coloring number of  uncountable graphs}
\author[T. Usuba]{Toshimichi Usuba}
\address[T. Usuba]
{Faculty of Science and Engineering,
Waseda University, 
Okubo 3-4-1, Shinjyuku, Tokyo, 169-8555 Japan}
\email{usuba@waseda.jp}
\keywords{Coloring number; Fodor-type reflection principle; List-crhomatic number; Reflection principle; Singular compactness}
\subjclass[2010]{03E05, 05C15,05C63}

\begin{document}

\begin{abstract}
We study the list-chromatic number and the coloring number of graphs, especially uncountable graphs.
We show that the coloring number of a graph coincides with
its list-chromatic number provided that the diamond principle holds.
Under the GCH assumption, we prove the singular compactness theorem for
the list-chromatic number.
We also investigate reflection principles for 
the list-chromatic number and the coloring number of graphs.
\end{abstract}

\maketitle
\section{Introduction}

Throughout this paper,  a \emph{graph} means a non-directed simple graph,
that is, a graph $X$ is a pair $\seq{\calV_X,\calE_X}$ where
$\calV_X$ is the set of vertexes and $\calE_X \subseteq [\calV_X]^2$ the set of edges.
We frequently identify $\calV_X$ with the graph $X$,
and if $X$ is clear from the context,
$\calE_X$ is denoted as $\calE$ for simplicity.
The cardinality of the graph $X$, denoted by $\size{X}$, is the cardinality of the vertex set $\calV_X$.
For a graph $X=\seq{X, \calE}$ and $x, y \in X$,
when $\{x,y\} \in \calE $ we write  $x \relE y$ or $y \relE x$.   
For $x \in X$, let $\calE^x=\{y \in X\mid y \relE x\}$.
%We also let $\calE^x=\calE^x_X$.

\begin{define}
Let $X=\seq{X,\calE}$ be a graph.
\begin{enumerate}
\item A \emph{coloring of $X$} is a function on $X$.
A \emph{good coloring of $X$} is a coloring 
$f:X \to \mathrm{ON}$ such that
whenever $x\relE y$, we have $f(x) \neq f(y)$.
% \item The \emph{chromatic number of $X$}, $\Chr(X)$, is 
% the minimal  finite or infinite cardinal $\ka$ such that there is a good coloring $f:X \to \ka$.
\item For a finite or an infinite cardinal $\ka$, 
a \emph{$\ka$-assignment of $X$} is a function
$F:X \to [\mathrm{ON}]^\ka$.
\item The \emph{list-chromatic number of $X$}, $\List(X)$,
is the minimal finite or infinite cardinal $\ka$
such that for every $\ka$-assignment $F:X \to [\mathrm{ON}]^\ka$,
there is a good coloring $f$ of $X$ with $f(x) \in F(x)$.
% \item $\List^*(X)=
% \min\{\ka\mid $ for every $F:V \to [\ka]^\ka$,
% there is a good coloring $f$ with $f(x) \in F(x)\}$.
\item The \emph{coloring number of $X$}, $\col(X)$,
is the minimal finite or infinite cardinal $\ka$ such that $X$ admits 
a well-ordering $\triangleleft$
such that for every $x \in X$, we have $\size{\{y \in \calE^x\mid y \mathrel \triangleleft x\}}<\ka$.
\end{enumerate}
\end{define}

We know that %$\Chr(X) \le %\List^*(X) \le
$\List(X) \le \col(X)\le\size{X}$.
The coloring number was introduced by Erd\H{o}s-Hajnal \cite{EH}.
The list-chromatic number was done in Vizing \cite{Vizing} and Erd\H{o}s-Rubin-Taylor \cite{ERT} independently,
and Komj\'ath \cite{Komjath} studied the list-chromatic number of infinite graphs
extensively. See also Komj\'ath's survey \cite{Komjath2} for these numbers.
In this paper, we study further combinatorial properties of  the list-chromatic number and the coloring number
of uncountable graphs,
and we also investigate reflection principles for these numbers.

By Komj\'ath's work, it turned out that the difference between the list-chromatic number and the coloring number of infinite graphs is sensitive.
While it is consistent that there is a graph $X$ with $\col(X)>\List(X) \ge \om$, 
Komj\'ath \cite{Komjath} constructed a model of $\ZFC$ in which
$\col(X)=\List(X)$ holds
for every graph $X$ with infinite coloring number.
We show that this situation follows from the diamond principle, which gives another simple proof of Komj\'ath's result.
\begin{thm}\label{thm5}
Suppose that for every regular uncountable cardinal $\ka$
and stationary $S \subseteq \ka$, $\diamondsuit(S)$ holds (e.g., assume $V=L$).
Then for every graph $X$, if $\col(X)$ is infinite then
$\col(X)=\List(X)$.
\end{thm}
%It is known that this assumption holds under $V=L$.

A graph $Y$ is called a \emph{subgraph} of the graph $X$ if
$\calV_Y \subseteq \calV_X$ and $\calE_Y \subseteq \calE_X$.
A subgraph $Y$ of $X$ is an \emph{induced subgraph}
if $\calE_Y=[\calV_Y]^2 \cap \calE_X$.
It is clear that if $Y$ is a subgraph of $X$ then
$\col(Y) \le \col(X)$ and $\List(Y) \le \List(X)$.

The coloring number has many useful properties, one of these is the
\emph{singular compactness}.
Shelah \cite{Shelah} showed that if $\size{X}$ is singular and $\col(Y) \le \la$ for every subgraph $Y$ of size $<\size{X}$,
then $\col(X) \le \la$.
Unlike the coloring number, one can prove that the singular compactness does not hold for the list-chromatic number
in general (see Section 2 below).
However, we prove that it is valid under the GCH assumption.
\begin{thm}\label{thm4}
Suppose $\ka$ is a  singular cardinal such that the set
$\{\mu<\ka \mid 2^\mu=\mu^+\}$ contains a club in $\ka$.
For every graph $X$ of size $\ka$ and  infinite cardinal $\la$,
if $\List(Y) \le \la$ for every subgraph $Y$ of size $<\ka$,
then $\List(X) \le \la$.
\end{thm}

In Sections \ref{sec4}--\ref{sec6},
we investigate reflection principles for the coloring number and the list-chromatic number.
\begin{define}
For an infinite cardinal $\la$, let 
$\Refl(\List, \la)$ be the assertion that for every graph $X$ of size $\le \la$,
if $\List(X)>\om$ then $X$ has a subgraph $Y$ of size $\om_1$ with $\List(Y)>\om$.
This is equivalent to uncountable compactness for the list-chromatic number: For a graph $X$ of size $\le \la$,
if $\List(Y) \le \om$ for every subgraph $Y$ of $X$ with size $\le \om_1$,
then $\List(X) \le \om$.
$\Refl(\List)$ is the assertion that $\Refl(\List, \la)$ holds for every cardinal $\la$.
We define $\Refl(\col,\la)$ and
$\Refl(\col)$ by replacing the list-chromatic number with the coloring number.
\end{define}

Such reflection principles and uncountable compactness are studied in various fields,
e.g, Balogh \cite{Balogh}, Fleissner \cite{Fle}, 
Fuchino et al.\ \cite{Fetal}, 
Fuchino-Rinot \cite{FR}, Fuchino-Sakai-Soukup-Usuba \cite{FSSU}, and
Todor\v cevi\'c \cite{Tod, Tod2}.

Fuchino et al.\ \cite{Fetal} introduced the Fodor-type Reflection Principle $\FRP$,
which is a combinatorial principle and is consistent modulo large cardinal axiom.
$\FRP$ implies various reflection principles (\cite{Fetal}, \cite{FR}).
Fuchino-Sakai-Soukup-Usuba \cite{FSSU} proved that
$\FRP$ implies $\Refl(\col)$, actually $\FRP$ is 
equivalent to $\Refl(\col)$.
Other notable principle in this context is \emph{Rado's conjecture}, which is a reflection principle for the chromatic number of intersection graphs, and is consistent modulo large cardinal axiom (\cite{Tod}).
The reflection $\Refl(\col)$ is strictly  weaker than Rado's conjecture
since Rado's conjecture implies $\Refl(\col)$ but the converse does not hold (Fuchino-Sakai-Torres-Perez-Usuba \cite{FSTU}).

For  $\Refl(\List)$, Fuchino-Sakai \cite{FS} showed that
$\Refl(\List)$ holds after collapsing a supercompact cardinal to $\om_2$,
hence $\Refl(\List)$  is also consistent modulo large cardinal axiom.

It is known that $\Refl(\col)$, or even the local reflection $\Refl(\col, \la)$ for some $\la >\om_1$
is a large cardinal property;
If $\la>\om_1$ is regular and $\Refl(\col, \la)$ holds
then every stationary subset of $\la \cap \Cof(\om)$ is reflecting (see Fact \ref{non-ref ladder}).
In contrast with this result, we prove that
the local reflection $\Refl(\List, \la)$ for some fixed $\la$ is not a large cardinal property.
\begin{thm}\label{thm1}
Suppose GCH. Let $\la>\om_1$ be a cardinal,
and suppose $\mathrm{AP}_{\mu}$ holds (see Definition \ref{AP}) for every $\mu<\la$ with countable cofinality.
Then there is a poset $\bbP$ which is $\sigma$-Baire, satisfies $\om_2$-c.c.,
and forces that ``\,$\Refl(\List, \la)$ holds and $2^{\om_1}>\la$''.
\end{thm}

This theorem shows that the local reflection $\Refl(\List, \la)$ does not imply
$\Refl(\col,\la)$. 
On the other hand, 
%in the generic extension by $\bbP$ in this theorem, both the global reflections $\Refl(\col)$ and $\Refl(\List)$ may fail, and,
under the assumption that $\List(X)=\col(X)$ for every graph $X$ of size $\om_1$ with infinite coloring number, we have that $\Refl(\col, \la)$ implies $\Refl(\List,\la)$ for every $\la$,
in particular $\Refl(\col)$ implies $\Refl(\List)$.
This observation suggests a natural question:
%and Theorem \ref{thm5} still indicates some connection 
%between the coloring number and the list-chromatic number.
\begin{question}
Does the global reflection 
$\Refl(\col)$ imply $\Refl(\List)$?
How is the converse?
\end{question}
For this question, we show that 
$\Refl(\col)$ and $\Refl(\List)$ can be separated one from the other.
More precisely, we show that the global reflection does not imply 
the reflection at $\om_2$ for the other number.

\begin{thm}\label{thm2}
If $\ZFC+$``there exists a supercompact cardinal'' is consistent,
then the following theories are consistent as well:
\begin{enumerate}
\item $\ZFC+\Refl(\List)$ holds but $\Refl(\col, \om_2)$ fails.
\item $\ZFC+\Refl(\col)$ holds but $\Refl(\List, \om_2)$ fails.
\end{enumerate}
\end{thm}
Notice that, in \cite{FS}, they already showed the consistency of the theory
that $\ZFC+$``$\Refl(\col)$ holds but $\Refl(\List, \om_3)$ fails''.

This paper is organized as follows.
In Section \ref{pre}, we present some basic definitions, facts, and easy observations.
We study combinatorial properties about the list-chromatic  number and the coloring number in Section \ref{sec3},
 and we prove Theorems \ref{thm5} and \ref{thm4}.
In  Section \ref{sec4} we discuss  basic results about reflection principles
and we prove Theorem \ref{thm2} (2).
In Section \ref{sec5}, we construct a forcing notion, and in Section \ref{sec6},
we prove Theorems \ref{thm1} and \ref{thm2} (1) using the forcing notion.

\section{Preliminaries}\label{pre}
We present some definitions, facts, and  easy observations,
which will be used later.

$\ka$, $\la$, $\mu$ will denote infinite cardinals unless otherwise specified.
$\mathrm{ON}$ is the class of all ordinals.
For an ordinal $\alpha$, let $\Cof(\alpha)$ be the class of ordinals with cofinality $\alpha$.

For a regular uncountable cardinal $\theta$,
let $\calH_\theta$ be the set of all sets with hereditary cardinality $<\theta$.

Let $\ka$ be a regular uncountable cardinal.
A stationary set $S \subseteq \ka$ is \emph{reflecting}
if there is some $\alpha<\ka$ such that
$S \cap \alpha$ is stationary in $\alpha$.
If $S$ is not reflecting, then $S$ is \emph{non-reflecting}.

For a regular uncountable cardinal $\ka$ and a stationary set $S \subseteq \ka$,
a sequence $\seq{d_\alpha \mid \alpha \in S}$ is a \emph{$\diamondsuit(S)$-sequence}
if for every $A \subseteq \ka$,
the set $\{ \alpha \in S \mid d_\alpha=A \cap \alpha\}$ is stationary in $\ka$.
Let us say that $\diamondsuit(S)$ holds if there is a $\diamondsuit(S)$-sequence.

\begin{fact}[Shelah \cite{S diamond}]\label{sdiamond}
Let $\ka$ be an infinite cardinal.
Suppose $2^\ka=\ka^+$.
Then for every stationary subset $S \subseteq \ka^+ \setminus \Cof(\cf(\ka))$,
$\diamondsuit(S)$ holds.
\end{fact}

\begin{define}
% \begin{enumerate}
% \item For a graph $X=\seq{X,E}$ and $x \in X$, let $\calE^x=\{y \in X\mid x E y\}$
% (we denote $\calE^x$ as $\calE^x_X$ if $X$ is not clear from the context).
%\item 
For an infinite set $A$ and a limit ordinal $\delta$, a \emph{filtration of $A$} is a $\subseteq$-increasing 
continuous sequence $\seq{A_\alpha\mid \alpha<\delta}$ such that
$\size{A_\alpha}<\size{A}$ for $\alpha<\delta$ and $\bigcup_{\alpha<\delta} A_\alpha=A$.
%\end{enumerate}
\end{define}

%We do not use the following fact in this note.

The following characterization of the coloring number is very useful.

\begin{fact}[Erd\H os-Hajnal \cite{EH}]\label{2.2+}
Let $X$ be a graph and $\ka$ an infinite cardinal.
Then the following are equivalent:
\begin{enumerate}
\item $\col(X) \le \ka$. 
\item There is a function $f:X \to [X]^{<\ka}$ such that
for every $x, y \in X$, if $x \relE y$ then either $x \in f(y)$ or $y \in f(x)$.
\item  There is a filtration $\seq{X_i\mid i<\delta}$ of
$X$ such that for every $i<\delta$, identifying $X_i$ with a induced subgraph of $X$,
$\col(X_i) \le \ka$, and,  for every $i<\delta$ and $x \in X \setminus X_i$,
we have $\size{X_i \cap \calE^x}<\ka$.
\item There is a 1-1 enumeration $\seq{x_i\mid i<\size{X}}$ of $X$
such that for every $i<\size{X}$,
we have $\size{ \{x_j\mid j<i, x_j \mathrel\calE x_i\}}<\ka$.
\end{enumerate}
\end{fact}

The following fact is immediate from the result in Komj\'ath \cite{Komjath}, and we 
present the proof for completeness.

\begin{fact}\label{0.3}
Let $X=\seq{X, \calE}$ be a graph and 
$\la$ and $\mu$ infinite cardinals with $\mu \le \la$.
If there are $Y_0, Y_1 \subseteq X$ such that
$\size{Y_0} = \la$, $\size{Y_1} \ge 2^\la$,
and $\size{\calE^z \cap Y_0} \ge \mu$ for every $z \in Y_1$, then
$\List(X)>\mu$.
\end{fact}
\begin{proof}
We may assume that $Y_0 \cap Y_1=\emptyset$.
It suffices to show that some subgraph of $X$ has  list-chromatic number $>\mu$.
By removing edges of $X$, we  may assume that 
$\size{\calE^z \cap Y_0} = \mu$ for every $z \in Y_1$.
Fix a pairwise disjoint family $\{A_y\mid y \in Y_0\}$ with $A_y \in [\la]^\mu$.
Since $\size{Y_1} \ge 2^\la$,
we can take an enumeration 
$\seq{g_z\mid z \in Y_1}$ of 
$\prod_{y \in Y_0} A_y$, possibly with repetitions.
Let $Y$ be the induced subgraph $Y_0 \cup Y_1$.
Define the $\mu$-assignment $F$ of $Y$ as follows.
For $z \in Y_1$, let $F(z)=\{g_z(y)\mid  y \relE z \land y \in Y_0\} \in [\la]^\mu$.
For $y \in Y_0$, let $F(y)=A_y$.
We show that there is no good coloring $f$ of $Y$ with $f(y) \in F(y)$.
For a coloring $f$ of $Y$ with $f(x) \in F(x)$,
there must be  $z \in Y_1$ with $f \restriction Y_0=g_z$.
Then $\{f(y)\mid y \relE z \land y \in Y_0 \}=\{g_z(y)\mid y \relE z \land y \in Y_0\}=F(z)$ and $f(z) \in F(z)$,
so $f$ is never good.
\end{proof}

For non-empty sets $X$ and $Y$,
let $K_{X,Y}$ be the complete bipartite graph on the bipartition classes $X$ and $Y$.
Fact \ref{0.3} yields the following fact. 
\begin{fact}[Lemma 6 in \cite{Komjath}]\label{2.2.4}
The complete bipartite graph $K_{\om, 2^{\om}}$ has
uncountable list-chromatic number.
\end{fact}

\begin{fact}[Lemma 9 in \cite{Komjath}]\label{2.2.5}
If $\la<2^\om$, then
the complete bipartite graph $K_{\om,\la}$ has
countable list-chromatic number.
In particular, every subgraph of $K_{\om, 2^\om}$ with size $<2^\om$ has
countable list-chromatic number.
\end{fact}

The following is also due to Komj\'ath.

\begin{fact}[\cite{Komjath}]\label{2.8+}
Let $X=\seq{\ka, \calE}$ be a graph on the regular uncountable cardinal $\ka$
and $\la<\ka$ an infinite cardinal.
Suppose $\col(Y) \le \la$ for every subgraph $Y$ of $X$ with size $<\ka$.
Then the following are equivalent:
\begin{enumerate}
\item The set $S=\{\alpha \in \ka \cap \Cof(\cf(\la)) \mid \exists \beta \ge \alpha\,(
\size{\alpha \cap \calE^\beta} \ge \la)\}$ is stationary in $\ka$.
\item The set $T=\{\alpha \in \ka \mid \exists \beta \ge \alpha\,(
\size{\alpha \cap \calE^\beta} \ge \la)\}$ is stationary in $\ka$.
\item $\col(X)>\la$.
\end{enumerate}
\end{fact}
\begin{proof}
(1) $\Rightarrow$ (2) is trivial.

(2) $\Rightarrow$ (3).
Suppose $T$ is stationary but $\col(X) \le \la$.
Then  by Fact \ref{2.2+}, 
there is $f:X \to [X]^{<\la}$ such that
whenever $\alpha, \beta \in X$ with $\alpha \relE \beta$, we have $\alpha \in f(\beta)$ or $\beta \in f(\alpha)$.
Since $T$ is stationary,
we can find $\alpha\in T$ such that
$f(\gamma) \subseteq \alpha$ for every $\gamma <\alpha$.
Fix $\beta \ge \alpha$ with
$\size{\calE^\beta \cap \alpha} \ge \la$.
Because $\size{f(\beta)}<\la$,
there is $\gamma \in \calE^\beta \cap \alpha$ with 
$\gamma \notin f(\beta)$.
$\beta$ is jointed with $\gamma$ but $\gamma \notin f(\beta)$,
so $\beta \in f(\gamma) \subseteq \alpha$.
This is a contradiction.

(3) $\Rightarrow$ (2).
Suppose $T$ is non-stationary, and 
then we can deduce $\col(X) \le \la$ as follows.
Fix a club $D$ in $\ka$ disjointing from
the set $\{\alpha \in \ka \mid 
 \exists \beta \ge \alpha\,(
\size{\alpha \cap \calE^\beta} \ge \la)\}$.
The sequence $\seq{\alpha \mid \alpha \in D}$ is a filtration of $X$.
Moreover $\col(\alpha) \le \la$ 
and $\size{\alpha \cap \calE^\beta}<\la$ for every $\alpha \in D$ and $\beta \in \ka \setminus \alpha$.
Applying Fact \ref{2.2+}, we have $\col(X) \le \la$.

(2) $\Rightarrow$ (1).
Suppose $S$ is non-stationary.
Let $T'=\{\alpha \in \ka \setminus \Cof(\cf(\la))\mid 
 \exists \beta \ge \alpha\,(
\size{\alpha \cap \calE^\beta} \ge \la)\}$.
We will show that $T'$ is non-stationary, then we can conclude that $T$ is non-stationary.
Suppose to the contrary that 
$T'$ is stationary.
For $\alpha \in T'$, fix $\beta \ge \alpha$ with $\size{\calE^\beta \cap \alpha} \ge \la$.
Since $\cf(\alpha) \ne \cf(\la)$,
there is $g(\alpha)<\alpha$
with
$\size{\calE^\beta \cap g(\alpha)} \ge \la$.
By Fodor's lemma,
there is $\gamma<\ka$ such that
the set $\{\alpha\in T'\mid g(\alpha)=\gamma\}$
is stationary.
Now take an arbitrary $\alpha$ with $\cf(\alpha)=\cf(\la)$ and $\alpha>\gamma$.
Then we can take $\alpha' \in  T'$ with $\alpha'>\alpha$ and $g(\alpha')=\gamma$.
Then $\size{\calE^\beta \cap \gamma} \ge \la$ for some $\beta \ge \alpha'>\alpha$.
This means that $\alpha \in S$, 
so $S$ is stationary. This is a contradiction.
\end{proof}
%Note that 
%we can prove (2) $\Rightarrow$ (1) without the assumption that
%``$\col(Y) \le \la$ for every subgraph $Y$ of $X$ with size $<\ka$''.

The next fact is a consequence of Shelah's singular compactness theorem \cite{Shelah}.

\begin{fact}[Shelah \cite{Shelah}]\label{2.10++}
Let $X$ be a graph, and suppose that $\size{X}$ is a singular cardinal.
For an infinite cardinal $\la$,
if $\col(Y) \le \la$ for every subgraph $Y$ of $X$ with size $<\size{X}$,
then $\col(X) \le \la$.
\end{fact}

The singular compactness for the list-chromatic number does not hold in general;
Komj\'ath proved that 
$\List(K_{\om,2^\om})>\om$ but
if $H$ is a subgraph of $K_{\om,2^\om}$ and $\size{H}<2^\om$,
then $\List(H) \le \om$.
Hence if $2^\om$ is singular
then $K_{\om, 2^\om}$ exemplifies the failure of Shelah's singular compactness 
with respect  to the list-chromatic number.
On the other hand, in the next section,
we will show that
the singular compactness for the list-chromatic number holds
under GCH.

The singular compactness immediately yields the following, which we will use frequently.
\begin{cor}\label{2.10+++}
Let $\la$ be an infinite cardinal, and $X$ a graph with $\col(X)>\la$.
Then $X$ has a subgraph $Y$ such that 
$\size{Y}$ is regular uncountable, $\col(Y)>\la$, and $\col(Z) \le \la$ for every subgraph $Z$ of $Y$ with
size $<\size{Y}$.
\end{cor}
\begin{proof}
Let $\ka=\min\{\size{Y} \mid Y$ is a subgraph of $X$, $\col(Y)>\la\}$.
Clearly $\ka$ is uncountable.
Pick a subgraph $Y$ of $X$ with size $\ka$ and
$\col(Y) >\la$. By the minimality of $\ka$,
every subgraph of $Y$ with size $<\ka$ has coloring number $\le \la$.
Hence $\ka$ must be regular by Fact \ref{2.10++}.
\end{proof}

We will use the forcing method,
so we fix some basic notations and definitions.
For a poset $\bbP$ and $p, q \in \bbP$,
if $p \le q$ then $p$ is an \emph{extension} of $q$.
$p$ and $q$ are \emph{compatible} if there is $r \in \bbP$ which is a 
common extension of $p$ and $q$.

For a cardinal $\ka$, a poset $\bbP$ is \emph{$\ka$-Baire}
if for every family $\calF$ of open dense subsets of $\bbP$ with
$\size{\calF}<\ka$,
the intersection $\bigcap \calF$ is dense in $\bbP$.
A poset $\bbP$ is $\ka$-Baire if and only if
the forcing with $\bbP$ does not add new $<\ka$-sequences.
\emph{$\sigma$-Baire} means $\om_1$-Baire.

For an uncountable cardinal $\ka$, and non-empty sets $X$ and $Y$,
let $\mathrm{Fn}(X,Y, <\ka)$ be the poset of all partial functions from $X$ to $Y$ with size $<\ka$.
The ordering is the reverse inclusion.

%\begin{fact}
%Let $\ka$ be a regular uncountable cardinal,
%and $Y$ a set with $2 \le \size{Y}<\ka$.
%Then for every stationary $S \subseteq \ka$,
%$\mathrm{Fn}(X,Y, <\ka)$ forces $\diamondsuit(S)$.
%\end{fact}

For a regular cardinal $\la$ and a set $X$ of ordinals,
let $\coll(\la,X)$ be the poset of all functions $p$ with size $<\la$
such that $\dom(p) \subset \la \times X$
and $p(\alpha,\beta) \in \beta$ for every $\seq{\alpha,\beta} \in \dom(p)$.
The ordering is given by the reverse inclusion.
$\coll(\la,X)$ is $\la$-closed, and forcing with
$\coll(\la, X)$ adds a surjection from $\la$ onto $\beta$ for every $\beta \in X$.
If $X$ is a regular cardinal $\ka$,
$\coll(\la, \ka)$ is denoted as $\coll(\la, <\ka)$,
and if $\ka$ is inaccessible, then $\coll(\la, <\ka)$ satisfies the $\ka$-c.c.
and forces $\ka=\la^+$.
%\begin{fact}
%\begin{enumerate}
%\item If $\ka$ is inaccessible, then $\coll(\la, <\ka)$ has the $\ka$-c.c. and
%forces that $\la^+=\ka$.
%\item For an ordinal $\alpha<\beta$,
%$\coll(\la, \beta)$ is isomorphic to $\coll(\la, \alpha) \times \coll(\la,[\alpha,\beta))$.
%\item Let $\bbP$ be a $\la$-closed poset,
%and $\alpha>2^{\size{\bbP}}$.
%Then $\coll(\la, \alpha)$ is forcing equivalent to 
%$\coll(\la, \alpha) \times \bbP$.
%\end{enumerate}
%\end{fact}

%\begin{fact}
%Let $\ka$ be an inaccessible. For every stationary $S \subseteq \la$,
%$\coll(\la, <\ka)$ forces $\diamondsuit(S)$.

%\end{fact}

A poset $\bbP$ is \emph{$\om_1$-stationary preserving} if
for every stationary set $S \subseteq \om_1$,
$\bbP$ forces that ``$S$ remains stationary in $\om_1$''.
\begin{lemma}\label{1.11}
Let $X$ be a graph.
\begin{enumerate}
\item If $\size{X}=\om_1$ and $\col(X)>\om$,
then every $\om_1$-stationary preserving poset
 forces that $\col(X) > \om$.
\item Suppose $\col(X) \ge \om_2$.
Then every $\om_2$-c.c. forcing notion forces
that $\col(X) \ge \om_2$.
\end{enumerate}
\end{lemma}
\begin{proof}
(1) is immediate from Fact \ref{2.8+} and the $\om_1$-stationary preservingness of $\bbP$.

For (2), take a subgraph $Y$ of $X$ such that
$\size{Y}$ is regular uncountable,
$\col(Y) >\om_1$, and $\col(Z) \le \om_1$ for every subgraph $Z$ of $Y$ with size $<\size{Y}$.
Let $\ka=\size{Y}$. Clearly $\ka \ge \om_2$.
We may assume $Y$ is of the form $\seq{\ka, \calE}$.
By Fact \ref{2.8+}, the set $S=\{\alpha<\ka \mid \exists \beta \ge \alpha\,(
\size{\calE^\beta \cap \alpha}  \ge \om_1)\}$ is stationary in $\ka$.
Since $\bbP$ satisfies the $\om_2$-c.c.,
we know that $S$ is stationary in $V^{\bbP}$.
Then the stationarity of $S$ witnesses that $\col(Y) \ge \om_2$,
hence so does $\col(X) \ge \om_2$ in $V^{\bbP}$.
\end{proof}

We will also use the Fodor-type reflection principle $\FRP$, which was introduced in
Fuchino et al. \cite{Fetal}.

\begin{define}[\cite{Fetal}]
For a regular $\ka \ge \om_2$,
$\FRP(\ka)$ is the assertion that
for every stationary $E \subseteq \ka \cap \Cof(\om)$ and
$g:E \to [\ka]^\om$ with $g(\alpha) \in [\alpha]^\om$,
there is $I \in [\ka]^{\om_1}$ such that
$\sup(I) \notin I$, $\cf(\sup(I))=\om_1$, and 
the set $\{x \in [I]^\om\mid \sup(x) \in E, g(\sup(x)) \subseteq x\}$ is stationary in $[I]^\om$.
$\FRP$ is the assertion that $\FRP(\ka)$ holds for every regular $\ka \ge \om_2$.
\end{define}

\begin{fact}[\cite{Fetal}, \cite{FSSU}]\label{basic FRP}
\begin{enumerate}
\item $\Refl(\col)$ holds if and only if
$\FRP$ holds.
\item $\FRP$ is preserved by any c.c.c. forcing.
\item If $\ka$ is supercompact, then $\coll(\om_1, <\ka)$ forces $\FRP$,
hence also $\Refl(\col)$.
\end{enumerate}
\end{fact}
See also Corollary \ref{frp to col}, which provides the proof of
$\FRP \Rightarrow \Refl(\col)$.
For completeness, let us sketch the proof of (3).
\begin{proof}
Take a $(V, \coll(\om_1, <\ka))$-generic $G$.
In $V[G]$, fix a regular cardinal $\la \ge \ka$.
We show that $\FRP(\la)$ holds in $V[G]$.
Take a stationary $E \subseteq \la \cap \Cof(\om)$,
and $g:E \to [\la]^\om$ with $g(\alpha) \subseteq \alpha$.
Let $D=\{x \in [\la]^\om \mid \sup(x) \in E, g(\sup(x)) \subseteq x\}$.
$D$ is stationary in $[\la]^\om$.

In $V$, take a $\la$-supercompact elementary embedding $j:V \to M$ with
critical point $\ka$, that is, ${}^\la M \subseteq M$
and $\la<j(\ka)$.
Then take a $(V[G], \coll(\om_1, [\ka, j(\ka))))$-generic $G_{tail}$.
We can take a $(V, \coll(\om_1, <j(\ka)))$-generic $j(G)$
with $j(G) \cap \coll(\om_1, <\ka)=G$ and $V[j(G)]=V[G][G_{tail}]$.
Since $\coll(\om_1, [\ka, j(\ka)))$ is $\sigma$-closed in $V[G]$,
we have that $M[j(G)]$ is closed under $\om$-sequences and 
$D$ remains stationary in $V[j(G)]$.
In addition one can check that $j``D=\{x \in [j``\la]^\om \mid \sup(x) \in j(E), j(g)(\sup(x)) \subseteq x\}$,
$j``D \in M[j(G)]$, $\size{j``\la}=\om_1$, and 
$j``D$ is stationary in $[j``\la]^\om$ in $M[j(G)]$.
Hence in $M[j(G)]$, 
$j``\la$ and $j``D$ witness the statement
that ``there is $I \in [j(\la)]^{\om_1}$ such that
$\sup(I) \notin I$, $\cf(\sup(I))=\om_1$,
and $\{x \in [I]^\om \mid \sup(x) \in j(E), j(g)(\sup(x)) \subseteq x\}$ is stationary''.
By the elementarity of $j$,
it holds in $V[G]$ that
 ``there is $I \in [\la]^{\om_1}$ such that
$\sup(I) \notin I$, $\cf(\sup(I))=\om_1$,
and $\{x \in [I]^\om \mid \sup(x) \in E, g(\sup(x)) \subseteq x\}$ is stationary''.
\end{proof}

The following may be a kind of folklore,
and the author found the proof of it in \cite{FS}.

\begin{fact}
\label{non-ref ladder}
Let $\ka$ be a regular uncountable cardinal, and
suppose $S \subseteq \ka \cap \Cof(\om)$ is a non-reflecting stationary set.
Then there is a graph $X$ of size $\ka$ such that
$\col(X)>\om$ but $\col(Y) \le \om$ for every subgraph $Y$ of $X$ with size $<\ka$.
In particular, if $\Refl(\col, \ka)$ holds for some regular $\ka \ge \om_2$,
then every stationary subset of $\ka \cap \Cof(\om)$ is reflecting.
\end{fact}
\begin{proof}
For $\alpha \in S$, take a cofinal set $c_\alpha \subseteq \alpha$ in $\alpha$ with
order type $\om$ and $c_\alpha \cap S=\emptyset$.
The vertex set of the graph $X$ is $S \cup \bigcup\{c_\alpha \mid \alpha \in S\}$,
and the edge set $\calE$ is defined by $\beta \relE \alpha \iff \alpha \in S$ and $\beta \in c_\alpha$.
We check  that the graph $X$ is as required.

For proving $\col(X)>\om$, suppose to the contrary that $\col(X) \le \om$.
Then by Fact \ref{2.2+}, there is $f:X \to [X]^{<\om}$ 
such that whenever $\beta \relE \alpha$, we have $\alpha \in f(\beta)$ or $\beta \in f(\alpha)$.
For $\alpha \in S$, since $f(\alpha)$ is finite but $c_\alpha$ is infinite,
we can choose $\gamma_\alpha \in c_\alpha \setminus f(\alpha)$.
By Fodor's lemma, there is $\gamma<\ka$ such that
the set $\{\alpha \in S \mid \gamma=\gamma_\alpha\}$ is stationary.
However then $\{\alpha \in S \mid \gamma=\gamma_\alpha\} \subseteq f(\gamma)$, this is impossible.

Next, take $\delta<\ka$ such that
$X \cap \delta=(S \cap \delta) \cup \bigcup\{c_\alpha \mid \alpha \in S \cap \delta\}$.
 We show that $\col(X \cap \delta) \le \om$.
Since $S \cap \delta$ is non-stationary in $\delta$, we can find a function $g$ on $S \cap \delta$
such that $g(\alpha)<\alpha$ and $\{c_\alpha \setminus g(\alpha) \mid \alpha \in S \cap \delta\}$
is a pairwise disjoint family (e.g., see Lemma 2.12 in Eisworth \cite{Eisworth}).
Note that $g(\alpha) \cap c_\alpha$ is finite.
Define $f:X \cap \delta \to [X \cap \delta]^{<\om}$ as follows:
For $\alpha \in X \cap \delta$, if $\alpha \in S \cap \delta$ then
$f(\alpha)=g(\alpha) \cap c_\alpha$.
If $\alpha \notin S \cap \delta$ and $\alpha \in g(\alpha') \cap c_{\alpha'}$ for some $\alpha' \in S \cap \delta$,
then $f(\alpha)=\emptyset$.
If $\alpha \notin S \cap \delta$ but $\alpha \notin g(\alpha') \cap c_{\alpha'}$ for every $\alpha' \in S\cap \delta$,
there is a unique $\alpha^* \in S \cap \delta$ with
$\alpha \in c_{\alpha^*} \setminus g(\alpha^*)$.
Then let $f(\alpha)=\{\alpha^*\}$.
One can check that $\alpha \relE \beta \Rightarrow \alpha \in f(\beta)$ or $\beta \in f(\alpha)$,
hence $\col(X \cap \delta) \le \om$ by Fact \ref{2.2+}.
\end{proof}

\section{Combinatorial results}\label{sec3}
In this section, we present some combinatorial results about the list-chromatic number  and the coloring number.

By the forcing method, Komj\'ath \cite{Komjath} showed the consistency 
of the statement that for every graph $X$, if $\col(X)$ is infinite  then
$\List(X)=\col(X)$.
We give another proof of Komj\'ath's result, in fact
the diamond principle is sufficient to obtain it.

\begin{prop}\label{2.7+}
Let $\ka$ be a regular uncountable cardinal  and $\la<\ka$ an infinite cardinal.
Suppose $\diamondsuit(S)$ holds for every stationary $S \subseteq \ka \cap \Cof(\cf(\la))$.
Then for every graph $X=\seq{\ka, \calE}$,
if  $\col(X)>\la$ but $\col(Y) \le \la$ for every subgraph $Y$ of $X$ with size $<\ka$,
then $\List(X)>\la$.
\end{prop}
\begin{proof}
Let $S=\{\alpha \in \ka \cap \Cof(\cf(\la))\mid  \exists \beta \ge \alpha\,(
\size{\alpha \cap \calE^\beta} \ge \la)\}$.
$S$ is stationary in $\ka$ by Fact \ref{2.8+}.
For each $\alpha \in S$,
take $\beta(\alpha) \ge \alpha$ with
$\size{\alpha \cap \calE^{\beta(\alpha)}} \ge \la$.
If necessary, by shrinking $S$ we may assume that 
for every $\alpha, \alpha' \in S$, if $\alpha<\alpha'$ then $\beta(\alpha)<\beta(\alpha')$.

By our assumption, $\diamondsuit(S)$ holds.
Then by a standard  coding argument,
there is a sequence $\seq{d_\alpha, e_\alpha\mid \alpha \in S}$ such that
$d_\alpha, e_\alpha:\alpha \to \alpha$ and
for every $f, g :\ka \to \ka$,
the set  $\{\alpha \in S \mid f \restriction \alpha=d_\alpha, g \restriction \alpha=e_\alpha\}$
is stationary in $\ka$.

Fix a pairwise disjoint sequence $\seq{A_\beta\mid \beta<\ka}$ with $A_\beta \in [\ka]^\la$.
For each $\alpha \in S$,
fix a set $x_\alpha \subseteq \alpha \cap \calE^{\beta(\alpha)}$ with $\size{x_\alpha} = \la$.
Now we define two $\la$-assignments $F, G$ in the following manner.
For $\beta <\ka$,
\begin{enumerate}
\item If $\beta \neq \beta(\alpha)$ for every $\alpha \in S$, then $F(\beta)=G(\beta)=A_\beta$.
\item Suppose $\beta=\beta(\alpha)$ for some (unique) $\alpha \in S$.
\begin{enumerate}
\item If $\size{d_\alpha``x_\alpha}=\la$,
then $F(\beta)=d_\alpha``x_\alpha$ and $G(\beta)=A_\beta$.
\item If $\size{d_\alpha``x_\alpha}<\la$ and
$\size{e_\alpha``x_\alpha}=\la$,
then $F(\beta)=A_\beta$ and $G(\beta)=e_\alpha``x_\alpha$.
\item If $\size{d_\alpha``x_\alpha}, \size{e_\alpha``x_\alpha}<\la$,
let $F(\beta)=G(\beta)=A_\beta$.
\end{enumerate}
\end{enumerate}
Note that if $F(\beta) \neq A_\beta$, then
the condition (2)(a) must be applied to $\beta$, so $G(\beta)$ is $A_\beta$.

Take arbitrary colorings $f, g :\ka \to \ka$ with
$f(\beta) \in F(\beta)$ and $g(\beta) \in G(\beta)$.
We show that if $f$ is good then $g$ is not good.

Now suppose $f$ is good.
Take $\alpha \in S$ with $f \restriction \alpha=d_\alpha$ and $g \restriction \alpha=e_\alpha$.
We know $f \restriction x_\alpha=d_\alpha \restriction x_\alpha$
and $g \restriction x_\alpha=e_\alpha \restriction x_\alpha$.
Let $\beta=\beta(\alpha)$.
If $\size{d_\alpha``x_\alpha}=\lambda$,
then $F(\beta)=d_\alpha``x_\alpha=f``x_\alpha$.
Since $f(\beta) \in F(\beta)$, there is $\eta \in x_\alpha$ with
$f(\eta)=f(\beta)$, this is a contradiction because $\eta \in x_\alpha \subseteq \calE^\beta$.
Hence we have $\size{d_\alpha``x_\alpha}<\la$.
In addition,
since the family $\seq{A_\beta \mid \beta<\ka}$ is pairwise disjoint,
we know that  the function $f \restriction \{\eta \in x_\alpha \mid F(\eta)=A_\eta\}$ is  
injective.
Thus we have that $\size{\{\eta \in x_\alpha \mid F(\eta) = A_\eta\}}<\la$,
otherwise we have $\size{d_\alpha``x_\alpha}=\la$.
Hence the set $\{\eta \in x_\alpha \mid F(\eta) \neq A_\eta\}$ has cardinality $\la$.
For $\eta \in x_\alpha$ with $F(\eta) \neq A_\eta$,
we know that $G(\eta)=A_\eta$.
By the same reason before, the map $g \restriction \{\eta \in x_\alpha \mid F(\eta) \neq A_\eta\}$ 
is injective, and we have $\size{e_\alpha``x_\alpha}=\size{g ``x_\alpha}=\la$.
Then $G(\beta)=e_\alpha``x_\alpha$,
and we can find $\eta \in x_\alpha$ with $g(\eta)=g(\beta)$.
Therefore $g$ is not good.
\end{proof}

\begin{cor}\label{2.10+}
Let $\ka$ be an uncountable cardinal and $\la <\ka$ an infinite cardinal.
Suppose that for every regular uncountable $\mu<\ka$ and every 
stationary $S \subseteq \mu \cap \Cof(\cf(\la))$,
$\diamondsuit(S)$ holds.
Then for every graph $X$ of size $<\ka$,
$\col(X)>\la$ if and only if $\List(X)>\la$.
\end{cor}
\begin{proof}
Take a graph $X$ of size $<\ka$ and
$\col(X)>\la$.
We shall prove that $\List(X)>\la$.
By Corollary \ref{2.10+++},
there is a subgraph $Y$ of $X$
such that $\size{Y}$ is regular uncountable, $\col(Y)>\la$,
and $\col(Z) \le \la$ for every subgraph $Z$ of $Y$ with $\size{Z}<\size{Y}$.
Then, by Proposition \ref{2.7+}, we have $\List(Y) >\la$,
hence $\List(X)>\la$.
\end{proof}

% So under $V=L$, we have
% $\Refl_{\List}=\Refl_{\col}=\infty$.

Now we have Theorem \ref{thm5}.
\begin{cor}\label{3.3++}
Suppose that for every regular uncountable $\ka$ and stationary $S \subseteq \ka$,
$\diamondsuit(S)$ holds.
Then for every graph $X$,
if $\col(X)$ is infinite then
$\col(X)=\List(X)$.
\end{cor}
\begin{proof}
Komj\'ath \cite{Komjath} proved 
that if $\col(X)=\om$, then $\List(X)=\om$.
The case $\col(X)>\om$ follows from Corollary \ref{2.10+}.
\end{proof}

\begin{cor}\label{3.4.1}
If $\diamondsuit(S)$ holds for every stationary $S \subseteq \om_1$,
then every graph of size $\om_1$ with uncountable coloring number has uncountable
list-chromatic number.
\end{cor}

Next we prove the list-chromatic version of (3) $\Rightarrow$ (1) in Fact \ref{2.2+}.

\begin{lemma}\label{0.4}
Let $X$ be a graph and $\la$ an infinite cardinal.
Suppose there is a filtration $\seq{X_\alpha\mid \alpha<\delta}$ of $X$
such that for every $\alpha<\delta$ and $x \in X \setminus X_\alpha$,
we have $\List(X_\alpha) \le \la$ and $\size{\calE^x \cap X_\alpha}<\la$.
Then $\List(X) \le \la$.
\end{lemma}
\begin{proof}
Fix a $\la$-assignment $F$ of $X$.
We construct a good coloring $f$ of $X$ with $f(x) \in F(x)$.
We do this by induction on $\alpha < \delta$.
Let  $\alpha < \delta$ and 
suppose $f \restriction X_\beta$ is defined to be a good coloring with 
$f(x) \in F(x)$ for every $\beta<\alpha$.
If $\alpha$ is limit, then let $f\restriction X_\alpha=\bigcup_{\beta<\alpha} f \restriction X_\beta$.
Suppose $\alpha=\gamma+1$.
Consider the induced subgraph $Y=X_\alpha \setminus X_\gamma$.
Note that $\List(Y) \le \List(X_\alpha)\le \la$.
By our assumption, 
for every $x  \in Y$, we have that $\calE^x \cap X_\gamma$ has cardinality $<\la$.
Hence $F'(x)=F(x) \setminus (f``(\calE^x \cap X_\gamma))$ has cardinality $\la$, and
$F'$ is a $\la$-assignment of $Y$.
Since $\List(Y) \le \la$, there is a good coloring $f'$
of $Y$  with $f'(x) \in F'(x)$.
Now let $f\restriction X_\alpha=(f \restriction X_\gamma) \cup f'$.
Finally, $f=\bigcup_{\alpha<\delta} f \restriction X_\alpha$ is a good coloring of $X$ with $f(x) \in F(x)$.
\end{proof}

As stated before, the singular compactness for the list-chromatic number does not 
hold in general.
On the other hand, the singular compactness can  hold under a certain cardinal arithmetic assumption.
The following proposition is Theorem \ref{thm4}.
\begin{prop}\label{1.3}
Let $\ka$ be a strong limit singular cardinal
such that there is a club $C$ in $\ka$ such that 
$2^\mu=\mu^+$ for every $\mu \in C$.
Let  $\la$ be  an infinite cardinal with $\la<\ka$.
For every graph $X$ of size $\ka$,
if $\List(Y) \le \la$ for every subgraph $Y$ of size $<\ka$,
then $\List(X) \le \la$.
\end{prop}
\begin{proof}
Fix an increasing continuous sequence $\seq{\ka_i\mid i<\cf(\ka)}$ with
limit $\ka$ such that  $\cf(\ka)+\la<\ka_0$ and
$2^{\ka_i}=\ka_i^+$ for $i<\cf(\ka)$.

First we claim:
\begin{claim}
For every $i<\cf(\ka)$ and subset $Y$ of $X$ with size $\ka_i$,
the set $\{x \in X\mid \size{\calE^x \cap Y} \ge \la\}$ has cardinality at most $\ka_i$.
\end{claim}
\begin{proof}
Suppose to the contrary that the set $\{x \in X\mid \size{\calE^x \cap Y} \ge \la\}$ has cardinality $>\ka_i$.
There is $Z \in [X]^{\ka_i^+}$ such that $\size{\calE^x \cap Y} \ge \la$ for every $x \in Z$.
However then the induced subgraph $Y \cup Z$ has cardinality $<\ka$ but $\List(Y \cup Z)>\la$ by Fact \ref{0.3}
and the assumption that $2^{\ka_i}=\ka_i^+$,
this is a contradiction.
%
%Case 2: $\cf(\ka_i)\neq \cf(\la)$.
%Note that $2^{<\ka_i}=\ka_i$ in this case.
%Fix an enumeration $\{y_\alpha \mid \alpha<\ka_i\}$ of $Y$,
%and let $Y_\beta=\{y_\alpha \mid \alpha<\beta\}$ for $\beta<\ka_i$.
%For each $x \in X$ with $\size{\calE^x \cap Y} \ge \la$,
%because $\cf(\ka_i) \neq \cf(\la)$,
%we can find $\beta_x<\ka_i$ such that
%$\size{\calE^x \cap Y_{\beta_x}} \ge \la$.
%Then there is some $\beta$ such that
% the set $\{x \in X \mid \beta_x=\beta\}$ has cardinality at least $\ka_i^+$.
%Hence we can take a set $Z \subseteq X$ of size $\ka_i^+$ such that
%$\size{\calE^z \cap Y_\beta} \ge \la$ for every $z \in Z$.
%%By Fact \ref{0.3}, the induced subgraph $Y \cup Z$ has list-chromatic number  $>\la$.
%Since $2^{<\ka_i}=\ka_i>\size{Y_\beta}$, we have $2^{\size{Y_\beta}}<\ka_i <\size{Z}$.
%Hence by Fact \ref{0.3}, the induced subgraph $Y_\beta \cup Z$ has list-chromatic number $>\la$.
%So we conclude that $X$ has a subgraph $Y \cup Z$ with size $<\ka$ 
%but list-chromatic number$>\la$,
%this is a contradiction.
\end{proof}

Fix a sufficiently large regular cardinal $\theta$.
We can take a sequence $\langle M_i^\alpha\mid i<\cf(\ka), \alpha<\la^+\rangle$ such that:
\begin{enumerate}
\item $M_i^\alpha \prec \calH_\theta$, $\size{M_i^\alpha}=\ka_i \subseteq M_i^\alpha$, and
$M_i^\alpha$ contains all relevant objects.
\item For every $\alpha<\la^+$, $\seq{M_i^\alpha\mid i<\cf(\ka)}$ is $\subseteq$-increasing
and continuous.
\item For every $\alpha<\la^+$ and $i<\cf(\ka)$,
if $\alpha$ is limit then we have $M_i^\alpha=\bigcup_{\beta<\alpha} M_i^\beta$.
\item For every $\alpha<\la^+$,
$\seq{M_i^\alpha\mid i<\cf(\ka)} \in M_0^{\alpha+1}$.
\end{enumerate}
Note that for every $\alpha<\beta<\la^+$ and $i,j<\cf(\ka)$,
we have that $M_i^\alpha \in M_j^\beta$.

For $i<\cf(\ka)$, let $M_i=\bigcup_{\alpha<\la^+} M_i^\alpha$.
By the choice of the $M_i^\alpha$'s,
we know  that $\seq{M_i\mid i<\cf(\ka)}$ is $\subseteq$-increasing, continuous,
$\size{M_i}=\ka_i \subseteq M_i$, and $X \subseteq \bigcup_{i<\cf(\ka)} M_i$.
We also know that $M^\alpha_i \in M_i$ for every $\alpha<\la^+$ and $i<\cf(\ka)$.
Let $X_i=X \cap M_i$.
The sequence $\seq{X_i\mid i<\cf(\ka)}$ is a filtration of $X$.
Now we show that $\size{\calE^x \cap X_i} < \la$ for every $i<\cf(\ka)$ and $x \in X \setminus X_i$,
and then the assertion follows  from Lemma \ref{0.4}.
Suppose to the contrary that $\size{\calE^x \cap X_i} \ge \la$ for some $i<\cf(\ka)$ and $x \in X \setminus X_i$.
Since $X_i=\bigcup_{\alpha<\la^+} X \cap M_i^\alpha$,
there is some $\alpha<\la^+$ with
$\size{\calE^x \cap M_i^\alpha} \ge \la$.
We know $M_i^\alpha \in M_i$,
so $\{ y \in X \mid  \size{\calE^y \cap M^\alpha_i} \ge \la\} \in M_i$. 
By the claim,
the set $\{y \in X\mid  \size{\calE^y \cap M_i^\alpha} \ge \la \}$ has cardinality at most $\ka_i$.
Since $\ka_i \subseteq M_i$,
we have $\{y \in X\mid  \size{\calE^y \cap M_i^\alpha} \ge \la \} \subseteq M_i$.
Now $x \in \{y \in X\mid  \size{\calE^y \cap M_i^\alpha} \ge \la \}$, so $x \in M_i \cap X=X_i$,
this is a contradiction.
\end{proof}

%Theorem \ref{thm4} is immediate from the previous proposition:
%\begin{cor}
%Suppose $\ka$ is a  singular cardinal such that the set
%$\{\mu<\ka \mid 2^\mu=\mu^+\}$ contains a club in $\ka$.
%For every graph $X$ of size $\ka$ and  infinite cardinal $\la$,
%if $\List(Y) \le \la$ for every subgraph $Y$ of size $<\ka$,
%then $\List(X) \le \la$.
%\end{cor}

\begin{question}
Can we weaken the assumption in Proposition \ref{1.3}?
For instance, does the conclusion of Proposition \ref{1.3} hold if $\ka$ is singular and $\mu^\la < \ka$ for every $\mu<\ka$,
or $\ka$ is singular and $2^\la<\ka$?
\end{question}

The following is a partial answer to this question.
%\begin{prop}
%Let $\ka$ be a s strong limit singular cardinal and $\la$ an infinite cardinal with $\la<\ka$ and $\om<\cf(\ka) \le \cf(\la)$.
%Then for every graph $X$ of size $\ka$,
%if $\List(Y) \le \la$ for every subgraph $Y$ of size $<\ka$,
%then $\List(X) \le \la$.
%\end{prop}
%\begin{proof}
%Fix an increasing continuous sequence $\seq{\ka_i \mid i<\cf(\ka)}$ with limit $\ka$ such that
%$\ka_0>\la$, $2^{<\ka_i}=\ka_i$, and $\cf(\ka_i) \neq \cf(\la)$, this is possible since $\om<\cf(\ka) \le \cf(\la)$ and
%$\ka$ is strong limit.
%By the proof of Proposition \ref{1.3},
%it is sufficient to show that for every $i<\cf(\ka)$ and subset $Y$ of $X$ with size $\ka_i$,
%the set $\{x \in X \mid \size{\calE^x \cap Y} \ge \la\}$ has cardinality at most $\ka_i$.
%Suppose not. Fix an enumeration $\{y_\alpha \mid \alpha<\ka_i\}$ of $Y$,
%and let $Y_\beta=\{y_\alpha \mid \alpha<\beta\}$ for $\beta<\ka_i$.
%For each $x \in X$ with $\size{\calE^x \cap Y} \ge \la$,
%because $\cf(\ka_i) \neq \cf(\la)$,
%we can find $\beta_x<\ka_i$ such that
%$\size{\calE^x \cap Y_{\beta_x}} \ge \la$.
%Because $\ka_i^+$ is regular,
%there is some $\beta$ such that
% the set $\{x \in X \mid \beta_x=\beta\}$ has cardinality at least $\ka_i^+$.
%Since $2^{<\ka_i}=\ka_i>\size{Y_\beta}$, we have $2^{\size{Y_\beta}}<\ka_i$.
%Then, by Fact \ref{0.3},  we can conclude that $X$ has a subgraph $Y$ with size $<\ka$ but $\List(Y)>\la$ , this is a contradiction.
%\end{proof}

\begin{prop}
Let $\ka$ be a  singular cardinal with countable cofinality, and $\la$ an infinite cardinal with $\la<\ka$.
Suppose $\mu^\la<\ka$ for every $\mu<\ka$.
For every graph $X$ of size $\ka$,
if $\List(Y) \le \la$ for every subgraph $Y$ of size $<\ka$,
then $\List(X) \le \la$.
\end{prop}
\begin{proof}
Fix an increasing sequence $\seq{\ka_n \mid n<\om}$ with limit $\ka$
such that $\ka_n^\la=\ka_n$. This is possible by our assumption.
Fix a sufficiently large regular cardinal $\theta$.
Take a $\subseteq$-increasing sequence of elementary submodels $\seq{M_n \mid n<\om}$ such that
$X \in M_n \prec \calH_\theta$, $\size{M_n}=\ka_n \subseteq M_n$,
and $[M_n]^\la \subseteq M_n$.
Let $X_n=M_n \cap X$. We know that $\seq{X_n \mid n<\om}$ is a filtration of $X$.
It is sufficient to show that for every $n<\om$ and $x \in X \setminus X_n$,
we have $\size{\calE^x \cap X_n} <\la$.
Suppose not. Take a set $a \subseteq \calE^x \cap X_n$ with size $\la$.
By the choice of $M_n$, we have $a \in M_n$.
The set $A=\{y \in X \mid a \subseteq \calE^y\}$ is in $M_n$,
and $x \in A$.
If the set $A$ has cardinality $\ge 2^\la$,
by Fact \ref{0.3} $X$ has a subgraph with size $<\ka$ but list-chromatic number $>\la$,
this is impossible.
Hence $\size{A} <2^\la \le \ka_n^\la=\ka_n$,
and $A \subseteq M_n$  since $\ka_n \subseteq M$.
Then $x \in A \subseteq M_n$, this is a contradiction.
\end{proof}

Erd\H os and Hajnal \cite{EH} showed that
if $\col(X)$ is uncountable,
then for every $n<\om$, $X$ contains a copy of $K_{n, \om_1}$ as a subgraph.
We prove a variant of their result under some additional assumptions.

\begin{define}[Shelah]\label{AP}
Let $\la$ be a cardinal.
$\mathrm{AP}_\la$ is the principle which asserts that
there is a sequence $\seq{c_\xi\mid \xi<\la^+}$ such that:
\begin{enumerate}
\item $c_\xi \subseteq \xi$.
\item There is a club $C$ in $\la^+$ such that
for every $\alpha \in C$,
\begin{enumerate}
\item $c_\alpha$ is unbounded in $\alpha$ and $\ot(c_\alpha)=\cf(\alpha)$.
\item $\{c_\alpha \cap \xi\mid \xi<\alpha\} \subseteq \{c_\xi\mid \xi<\alpha\}$.
\end{enumerate}
\end{enumerate}
\end{define}
See Eisworth \cite{Eisworth} about $\mathrm{AP}_\la$.
We point out that $\mathrm{AP}_\la$ follows from the weak square principle at $\la$.
%
%\begin{fact}
%$\square^*_\la \Rightarrow \mathrm{AP}_\la$.
%\end{fact}

\begin{prop}\label{1.8}
Let $X$ be a graph with $\col(X)>\om_1$.
Suppose $\la^\om=\la$ for every regular uncountable $\la <\size{X}$ (in particular $2^\om=\om_1$),
and $\mathrm{AP}_\la$ holds for every singular cardinal $\la<\size{X}$ of countable cofinality.
Then $X$ contains a subgraph which is isomorphic to $K_{\om, \om_1}$, in particular $X$
contains a subgraph of size $\om_1$ which has uncountable list-chromatic number.
\end{prop}
\begin{proof}
Choose a subgraph $Y$ of $X$ such that
$\size{Y}$ is regular uncountable, $\col(Y)>\om_1$, and 
$\col(Z)  \le \om_1$ for every subgraph $Z$ of $Y$ with size $<\size{Y}$.
Let $\ka=\size{Y}$. 
$\ka$ is strictly greater than $\om_1$.
We may assume $Y=\seq{\ka, \calE}$.
Since $\col(Y)>\om_1$,
the set $S=\{\alpha<\ka\mid \exists \beta \ge \alpha\,(\size{\calE^\beta \cap \alpha} >\om)\}$
is stationary in $\ka$ by Fact \ref{2.8+}.
Fix a sufficiently large regular cardinal $\theta$.

Case 1: $\ka$ is  not the successor of a singular cardinal of countable cofinality.
Note that $\gamma^\om<\ka$ for every $\gamma<\ka$ in this case.

Take $M \prec \calH_\theta$ which contains all relevant objects and
$M \cap \ka \in S$.
Fix $\beta_0 \ge M \cap \ka$ such that 
$\calE^{\beta_0} \cap (M \cap \ka)$ is uncountable.
Then there is $\gamma< M \cap \ka$ such that
$\calE^{\beta_0} \cap \gamma$ is infinite.
We know $\gamma^\om<\ka$, hence $\gamma^\om <M \cap \ka$ and 
$[\gamma]^\om \subseteq M$.
Take $Y_0 \in [\gamma]^\om$ such that
$Y_0 \subseteq \calE^{\beta_0}$.
By the elementarity of $M$,
the set $\{\beta<\ka\mid Y_0 \subseteq \calE^\beta\}$ 
is unbounded in $\ka$.
Hence we can find $Z \in [\ka]^{\om_1}$ such that
$Y_0 \cap Z=\emptyset$ and $Y_0 \subseteq \calE^\beta$ for every $\beta \in Z$.
Then the induced subgraph $Y_0 \cup Z$ contains a copy of $K_{\om, \om_1}$.

Case 2: $\ka$ is the successor of a singular cardinal of countable cofinality,
say $\ka=\la^+$ with $\cf(\la)=\om$.
We have that $\mathrm{AP}_\la$ holds and $\gamma^\om<\la$ for every $\gamma<\la$.

Take $M \prec \calH_\theta$ which contains all relevant objects
and $M \cap \ka \in S$.
Take a sequence $\seq{c_\xi\mid \xi<\ka} \in M$ witnessing $\mathrm{AP}_\la$.
Let $\alpha =M \cap \ka$.
Then $\sup(c_\alpha)=\alpha$ and $\ot(c_\alpha)<\la$.
We also know $c_\alpha \cap \gamma \in M$ for every $\gamma<\alpha$,
because $\{c_\alpha \cap \gamma \mid \gamma<\alpha\} \subseteq \{c_\gamma \mid \gamma<\alpha\} \subseteq M$.
Take a sequence $\seq{\pi_\xi\mid \xi<\ka} \in M$ such that
each $\pi_\xi$ is a surjection from $\la$ to $\xi$.
Take also an increasing sequence $\seq{\la_n\mid n<\om}\in M$ with limit $\la$.
For $n<\om$, let $A_n=\bigcup\{\pi_\xi``\la_n\mid \xi \in c_\alpha\}$.
We have that $\size{A_n}<\la$ and $\bigcup_{n<\om} A_n=\alpha$.
Fix $\beta_0 \ge \alpha$ with $\size{\calE^{\beta_0} \cap \alpha}>\om$.
Then there is $n_0<\om$
such that $A_{n_0} \cap \calE^{\beta_0}$ is uncountable.
For $\gamma<\alpha$,
let $B_\gamma=\bigcup \{\pi_\xi``\la_{n_0} \mid \xi \in c_\alpha \cap \gamma\}$.
The sequence $\seq{B_\gamma\mid \gamma<\alpha}$ is $\subseteq$-increasing
and $\bigcup_{\gamma<\alpha} B_\gamma=A_{n_0}$.
Thus there is some $\delta<\alpha$ such that
$\calE^{\beta_0} \cap B_\delta$ is infinite.
Since $c_\alpha \cap \delta \in M$,
we have that $B_\delta \in M$.
$\size{B_\delta}<\la$, so we have that $[B_\delta]^\om \subseteq M$, and 
there is $Y_1 \in [B_\delta]^\om$
such that $Y_1 \in M$ and $Y_1 \subseteq \calE^{\beta_0}$.
The rest is the same as Case 1.
\end{proof}

\begin{remark}
Under the assumption of Proposition \ref{1.8},
a graph $X$ with $\col(X)>\om_1$ actually contains a copy of $K_{\om, \om_2}$.
However we do not use this result in this paper.
\end{remark}

In the proof of the previous proposition,
we used the cardinal arithmetic assumption and the principle $\mathrm{AP}_\la$.

\begin{question}
Are the assumptions in Proposition \ref{1.8} necessary?
\end{question}

\section{Reflections for the list-chromatic number and the coloring number}\label{sec4}
In this section, we consider  the reflection principles $\Refl(\List)$ and $\Refl(\col)$,
and, we  prove some results about reflections.
First we prove that $\Refl(\List)$ implies the Continuum Hypothesis.

\begin{lemma}\label{CH}
$\Refl(\List, 2^\om)$ implies that
$2^\om=\om_1$.
\end{lemma}
\begin{proof}
The complete bipartite graph $K_{\om,2^\om}$
has uncountable list-chromatic number by Fact \ref{2.2.4}.
However, if $2^\om>\om_1$, then every subgraph of $K_{\om,2^\om}$ of size $\om_1$
has countable list-chromatic number by Fact \ref{2.2.5}.
This contradicts the principle $\Refl(\List, 2^\om)$.
\end{proof}
$\FRP$ and $\Refl(\col)$ follow from Martin's Maximum (\cite{Fetal}).
However 
this lemma shows that $\Refl(\List)$ does not follow from Martin's Maximum and other forcing axioms which imply  $2^\om>\om_1$.

\begin{question}
Does $\Refl(\List, \om_2)$ imply $2^\om=\om_1$?
\end{question}

Now we have the following consistency result, which is (2) of Theorem \ref{thm2}.

\begin{cor}
It is consistent that
$\Refl(\col)$ holds but $\Refl(\List, \om_2)$ fails.
\end{cor}
\begin{proof}
Let $\ka$ be a supercompact cardinal,
and take a $(V, \coll(\om_1, <\ka))$-generic $G$.
In $V[G]$, $\FRP$ holds by Fact \ref{basic FRP}.
Now add $\om_2$ many Cohen (or any other) reals by c.c.c. forcing.
$\FRP$ is preserved by c.c.c. forcing,
so $\FRP$ still holds in the extension,
and we have $\Refl(\col)$. 
On the other hand, since $2^\om=\om_2$ in the extension,
we have that $\Refl(\List, \om_2)$ fails.
\end{proof}

On the other hand, by using Corollary \ref{3.4.1},  we have the following implication between $\Refl(\col)$ and $\Refl(\List)$.
\begin{cor}\label{4.5++}
Suppose that $\diamondsuit(S)$ holds for every stationary $S \subseteq \om_1$.
If $\Refl(\col)$ holds, then $\Refl(\List)$ holds as well.
\end{cor}
\begin{proof}
Let $X$ be a graph with $\List(X)>\om$.
Then $\col(X) \ge \List(X)>\om$,
hence $X$ has a subgraph $Y$ of size $\om_1$ with uncountable coloring number.
Now we have $\List(Y)>\om$ by Corollary \ref{3.4.1}.
\end{proof}

\begin{cor}[\cite{FS}]
Suppose $\ka$ is supercompact.
Then $\coll(\om_1, <\ka)$ forces $\Refl(\List)$.
\end{cor}
\begin{proof}
In $V^{\coll(\om_1, <\ka)}$, $\FRP$ holds and it is known that $\diamondsuit(S)$ holds for all stationary subsets $S$ in $\om_1$.
Then $\Refl(\List)$ holds as well by the previous corollary.
\end{proof}

Next we turn to the consistency strength of $\Refl(\List)$.

\begin{prop}\label{cons of list}
Suppose $\Refl(\List)$.
Then for every cardinal $\la \ge \om_1$ of uncountable cofinality,
either:
\begin{enumerate}
\item $2^\la>\la^+$, or
\item Every stationary subset of $\la^+ \cap \Cof(\om)$ is reflecting.
\end{enumerate}
\end{prop}
\begin{proof}
Let $\la \ge \om_1$ be a cardinal of uncountable cofinality.
If $2^\la=\la^+$,
then 
$\diamondsuit(S)$ holds for every stationary $S \subseteq \la^+ \cap \Cof(\om)$
by Fact \ref{sdiamond}.
If $\la^+ \cap  \Cof(\om)$ has a non-reflecting stationary subset,
by Fact \ref{non-ref ladder} there is a graph $X$ of size $\la^+$ such that
$\col(X) >\om$ but $\col(Y) \le \om$ for every subgraph $Y$ of size $<\la^+$.
Then $\List(X)>\om$ but $\List(Y) \le \om$ for every $Y \in [X]^{<\la^+}$ by
Proposition \ref{2.7+}, this is a contradiction.
\end{proof}

This proposition means that 
the global reflection $\Refl(\List)$ has a large cardinal strength;
The singular cardinal hypothesis fails, or
$\square_\la$ fails at every singular cardinal $\la$ of uncountable cofinality.

We will use the following proposition later.
\begin{prop}\label{1.7}
Suppose $\diamondsuit(S)$ holds for every stationary $S \subseteq \om_1$.
Let $\ka \ge \om_2$ be regular and
suppose $\FRP(\ka)$ holds.
Then for every graph $X$ of size $\ka$, if
$\List(Y) \le \om$ for every subgraph $Y$ of size $<\ka$,
then $\List(X) \le \om$.
\end{prop}
\begin{proof}
We may assume that the graph $X$ is of the form $\seq{\ka, \calE}$.
Let $S=\{\alpha \in \ka \cap \Cof(\om)\mid 
\exists \beta \ge \alpha\,(\calE^\beta \cap \alpha$ is infinite$)\}$.
If $S$ is non-stationary,
%as in the proof of (2) $\Rightarrow$ (1) in Fact \ref{2.8+}, 
%we have that $\{\alpha \in \ka \mid 
%\exists \beta \ge \alpha\,(\calE^\beta \cap \alpha$ is infinite$)\}$
%is non-stationary.
then we can deduce $\List(X) \le \col(X) \le \om$  by Fact  \ref{2.8+}.

Now we show that $S$ is non-stationary in $\ka$.
Suppose to the contrary that $S$ is stationary.
For each $\alpha \in S$, fix $\beta(\alpha) \ge \alpha$
such that $\calE^{\beta(\alpha)} \cap \alpha$ is infinite.
Let $C=\{\alpha<\ka\mid 
\beta(\alpha')<\alpha$ for every $\alpha' \in S \cap \alpha\}$.
$C$ is a club.
Let $S^*=S\cap C$. $S^*$ is also stationary.
Note that $\beta(\alpha) \neq \beta(\alpha')$ for every distinct $\alpha,\alpha' \in S^*$.
Take  $g:S^* \to [\ka]^\om$ such that
$g(\alpha) \in [\calE^{\beta(\alpha)} \cap \alpha]^\om$.
By $\FRP(\ka)$,
there is $I \in [\ka]^{\om_1}$ 
such that $\sup(I) \notin I$, $\cf(\sup(I))=\om_1$, and 
the set $A=\{x \in [I]^\om\mid  \sup(x) \in S^*, g(\sup(x)) \subseteq x\}$ is 
stationary in $[I]^\om$.
Let $Y$ be the induced subgraph $I \cup \{\beta(\alpha)\mid \alpha \in I \cap S^*\}$.
$\size{Y} =\om_1$, hence $\List(Y) \le \om$  by our assumption,
and $\col(Y) \le \om$ by 
Corollary \ref{3.4.1}.
By Corollary \ref{2.2+}, we can find $f:Y \to [Y]^{<\om}$ such that
for every distinct $\alpha,\alpha' \in Y$, if $\alpha \relE \alpha'$ then
either $\alpha \in f(\alpha')$ or $\alpha' \in f(\alpha)$.
For each $\alpha  \in Y$,
$g(\alpha)$ is infinite but $f(\beta(\alpha))$ is finite.
Thus we can take a function $h$ on $A$ so that $h(x) \in g(\sup(x)) \setminus f(\beta(\sup(x)))$.
Then, we can find $\gamma \in \bigcup A$ such that
$A'=\{x \in A\mid  h(x)=\gamma\}$
is stationary in $[I]^\om$.
For $x \in A'$,
since $\gamma=h(x) \in g(\sup(x)) \subseteq \calE^{\beta(\sup(x))} \cap \sup(x) $ but $\gamma \notin f(\beta(\sup(x)))$,
we have $\beta(\sup(x)) \in f(\gamma)$.
However this is impossible since $\{\beta(\sup(x))\mid x \in A'\}$ is infinite but $f(\gamma)$ is finite.
\end{proof}

The proof of Proposition \ref{1.7} yields the  following:
\begin{cor}[\cite{FSSU}]\label{frp to col}
If $\FRP$ holds, then $\Refl(\col)$ holds as well.
\end{cor}
\begin{proof}
By induction on size of graphs.
Let $X$ be a graph of size $\ge \om_2$, and
suppose every subgraph of size $\om_1$ has countable coloring number.
By the induction hypothesis,
every subgraph of size $<\size{X}$ has countable coloring number.
If $\size{X}$ is regular, argue as in the proof of Proposition \ref{1.7}.
If $\size{X}$ is singular, we can apply Fact \ref{2.10++}.
\end{proof}
As mentioned before,
in fact $\FRP$ is equivalent to $\Refl(\col)$
(\cite{FSSU}).

\section{Forcing notion adding a good coloring}\label{sec5}
In this section we define a forcing notion
which adds a good coloring of a given graph.
We will use this forcing notion for the proofs of Theorems \ref{thm1} and \ref{thm2} (1).

First we recall some basic definitions.
Let $\bbP$ be a poset.
Every set $x$ has the canonical name $\check x$
defined by $\check{x}=\{\seq{\check y,1} \mid y \in x\}$,
where $1$ is the maximum element of the poset.
We frequently omit the check of $\check{x}$,
and simply write $x$.

\begin{define}
Let $\bbP$ be a poset and $\theta$ a sufficiently large regular cardinal.
Let $M \prec \calH_\theta$ be a countable model with $\bbP \in M$.
\begin{enumerate}
\item A condition $p \in \bbP$ is an \emph{$(M, \bbP)$-generic condition}
if for every dense open set $D \in M$ in $\bbP$ and 
$q \le p$, there is $r \in D \cap M$ which is compatible with $q$.
\item A condition $p \in \bbP$ is a \emph{strong $(M,\bbP)$-generic condition}
if for every dense open set $D \in M$ in $\bbP$,
there is some $q \in D \cap M$ with $p \le q$.
\item A descending sequence $\seq{p_n\mid n<\om}$ in $\bbP$ is an
\emph{$(M,\bbP)$-generic sequence} if $p_n \in M$ for $n<\om$, and
for every dense open set $D \in M$ in $\bbP$,
there is $n<\om$ with $p_n \in D \cap M$.
\end{enumerate}
Every strong $(M,\bbP)$-generic condition is an $(M,\bbP)$-generic condition.
If an $(M,\bbP)$-generic sequence $\seq{p_n\mid n<\om}$ has a lower bound $p$,
then $p$ is a strong $(M, \bbP)$-generic condition.
\end{define}

Let $M \prec \calH_\theta$ be a model with $\bbP \in M$,
and $G$ be $(V, \bbP)$-generic.
Let $M[G]=\{\dot x_G \mid \dot x \in M\}$,
where $\dot x_G$ is the interpretation of $\dot x$ by $G$.
The following are known:
\begin{enumerate}
\item $M[G] \prec \calH_\theta^{V[G]}$.
\item If $G$ contains an $(M, \bbP)$-generic condition,
then $M \cap \mathrm{ON}=M[G] \cap \mathrm{ON}$.
\end{enumerate}

\begin{define}
Let $\bbP,\bbQ$ be posets,
and suppose $\bbP$ is a suborder of $\bbQ$,
that is, $\bbP \subseteq \bbQ$ and for $p_0, p_1 \in \bbP$, $p_0 \le p_1$ in $\bbP$ if and only if
$p_0 \le p_1$ in $\bbQ$.
\begin{enumerate}
\item For $q \in \bbQ$, a condition $p \in \bbP$
is a \emph{reduction of $q$} if
for every $r \le p$ in $\bbP$, $r$ is compatible with $q$ in $\bbQ$.
\item $\bbP$ is a \emph{complete suborder} of $\bbQ$
if (i) $p \bot q$ in $\bbP$ then so does in $\bbQ$,
and (ii) every $q \in \bbQ$ has a reduction $p \in \bbP$.
(ii) is equivalent to the property that every maximal antichain in $\bbP$ is maximal in $\bbQ$.
\item For a $(V,\bbP)$-generic $G$, the \emph{quotient poset} $\bbQ/G$
is the suborder $\{q \in \bbQ\mid q$ is compatible with any $p \in G\}$.
When $G$ is clear from the context,
$\bbQ/G$ is denoted by $\bbQ/\bbP$.
\end{enumerate}
\end{define}

\begin{fact}
Let $\bbQ$ be a poset and $\bbP$ a complete suborder of $\bbQ$.
\begin{enumerate}
\item If $G$ is $(V, \bbP)$-generic and $H$ is $(V[G], \bbQ/G)$-generic,
then $H$ is $(V, \bbQ)$-generic and $V[G][H]=V[H]$.
\item If $H$ is $(V, \bbQ)$-generic, then $G=H \cap \bbP$ is $(V, \bbP)$-generic,
$H$ is $(V[G],\bbQ/G)$-generic, and $V[H]=V[G][H]$.
%\begin{enumerate}
%\item For $p \in \bbP$ and $q \in \bbQ$, $p$ is a reduction of $q$ if and only if
%$p \Vdash_\bbP q \in \bbQ/\dot G$.
%\item Let $G$ be $(V, \bbP)$-generic. Then for $q \in \bbQ$,
%$q \in \bbQ/G$ if and only if there is a reduction $p \in \bbP$ of $q$
%with $p \in G$.
%\item
\item Suppose $q \in \bbQ$ has the greatest reduction $p \in \bbP$.
Then for every $(V, \bbP)$-generic $G$,
$q \in \bbQ/G$ if and only if $p \in G$.
\end{enumerate}
\end{fact}

In order to define our forcing notion, we need more definitions and lemmas.

\begin{define}
Let $\ka \ge \om_2$ be a cardinal.
We say that a graph $X$ is \emph{$\ka$-nice} if $X$ satisfies the following conditions:
\begin{enumerate}
\item $X$ is of the form $\seq{\ka, \calE}$.
\item $\col(X) \le \om_1$.
\item $\col(Y) \le \om$ for every subgraph $Y$ of size $\om_1$.
\item For every $\alpha<\ka$,
$\size{\calE^\alpha \cap \alpha} \le \om$.
\end{enumerate}
\end{define}
Notice that for a given graph $X$ of size $\ka$, if $\col(X) \le \om_1$, then 
there is an enumeration $\seq{x_i\mid i<\ka}$ of $X$ with
$\calE^{x_i} \cap \{x_j\mid j<i\}$ countable for every $i<\ka$ by Fact \ref{2.2+}.
Hence for every graph $X$ of size $\ka$, if $X$ satisfies the conditions (2) and (3)
then there is a $\ka$-nice graph which is isomorphic to $X$.

Fix a $\ka$-nice graph $X$ and an $\om$-assignment $F:\ka \to [\ka]^\om$.
Under CH, we shall define a forcing notion $\bbP$ such that
$\bbP$ satisfies the $\om_2$-c.c., $\sigma$-Baire,
and adds a good coloring $f$ of $X$ with $f(\alpha) \in F(\alpha)$.
Throughout this section, we assume CH. Let $\theta$ be a sufficiently large regular cardinal.

\begin{lemma}\label{2.1}
For every $x \in [\ka]^\om$,
the set $\{\beta<\ka\mid  \calE^\beta \cap x$ is infinite $\}$ is at most countable.
\end{lemma}
\begin{proof}
Otherwise,  we can find 
$Z \in [\ka]^{\om_1}$ such that
$\size{x \cap \calE^\beta} \ge \om$ for every $\beta \in Z$.
By CH and Fact \ref{0.3}, we have $\om<\List(x \cup Z)\le \col(x \cup Z)$, this  is a contradiction.
\end{proof}

\begin{define}\label{2.2}
A set $x \subseteq \ka$ is said to be \emph{$\seq{X,F}$-complete} (or simply \emph{complete}) if the following hold:
\begin{enumerate}
\item $\calE^{\alpha} \cap \alpha \subseteq x$ and $F(\alpha) \subseteq x$ for every $\alpha \in x$.
\item For every $\beta<\ka$, if $\calE^\beta \cap x$ is infinite then
$\beta \in x$.
\end{enumerate}
%Let $T^*$ be the set of all $x \in [\om_2]^\om$ with $\sup(x) \notin S^*$.
\end{define}
Note that if $x$ and $y$ are complete,
then both $x \cap y$ and $x \cup y$ are complete as well.

We say that a set $A$ is \emph{$\sigma$-closed} if
$[A]^\om \subseteq A$.
By CH, for each $x \in \calH_\theta$ there is a $\sigma$-closed $N \prec \calH_\theta$ of
size $\om_1$ containing $x$.

\begin{lemma}\label{2.3}
Let $N \prec \calH_\theta$ be  $\sigma$-closed with $X,F \in N$ and $\size{N}=\om_1$.
Then $N \cap \ka$ is complete.
\end{lemma}
\begin{proof}
It is clear that $\calE^\alpha \cap \alpha, F(\alpha) \subseteq N \cap \ka$ for every $\alpha \in N \cap \ka$.
Take $\beta<\ka$,
and suppose $\calE^\beta \cap (N \cap \ka)$ is infinite.
Take a countable subset $a \subseteq \calE^\beta \cap N \cap \ka$.
By the $\sigma$-closure of $N$,
we have that $a \in N$.
The set $\{\gamma<\ka \mid \calE^\gamma \cap a$ is infinite$\}$ is in $N$ and
at most countable by Lemma \ref{2.1}.
Hence $\{\gamma<\ka \mid \calE^\gamma \cap a$ is infinite$\}\subseteq N \cap \ka$,
and we have $\beta \in \{\gamma<\ka \mid \calE^\gamma \cap a$ is infinite$\}\subseteq N \cap \ka$.
\end{proof}

\begin{lemma}\label{2.3+}
Let $Y \subseteq \ka$ be a complete set of size $\om_1$,
and $M \prec \calH_\theta$ a countable model with $X, F, Y \in M$.
Then $M \cap Y$ is complete.
In particular, the set of all countable complete subsets of $\ka$ is stationary in $[\ka]^\om$.
\end{lemma}
\begin{proof}
By the assumption, we have $\col(Y) \le \om$.
Hence we can find a function $f:Y \to [Y]^{<\om}$ in $M$
such that for every $\alpha,\beta \in Y$, if $\alpha \relE \beta$ the
$\alpha \in f(\beta)$ or $\beta \in f(\alpha)$ by Fact \ref{2.2+}.

To show that $M \cap Y$ is complete,
we only show that for every $\beta<\ka$, if $M \cap Y \cap \calE^\beta$ is 
infinite then $\beta \in M \cap Y$.
Because $Y$ is complete, we have $\beta \in Y$.
On the other hand, since $M \cap Y \cap \calE^\beta$ is infinite but $f(\beta)$ is finite,
we can take  $\alpha \in (M \cap Y \cap \calE^\beta) \setminus f(\beta)$.
Then $\beta \in f(\alpha) \subseteq M$, so $\beta \in M \cap Y$ as required.
\end{proof}

Now we are ready to define our forcing notion.
 
\begin{define}\label{2.4}
$\bbP(X, F)$ is the poset which consists of 
all countable functions $p$ such that:
\begin{enumerate}
\item $p$ is a good coloring of the induced subgraph $\dom(p) \in [\ka]^\om$ with
$p(\alpha) \in F(\alpha)$.
%\item $F(\alpha) \subseteq \dom(p)$ and $\calE^\alpha \cap \alpha \subseteq \dom(p)$ 
%for every $\alpha \in \dom(p)$.
\item  $\dom(p)$ is complete.
\end{enumerate}
Define $p \le q$ if $p \supseteq q$.
\end{define}

For simplicity, we omit the parameters $X$ and $F$ in $\bbP(X,F)$ and just write $\bbP$.

\begin{lemma}\label{2.5}
\begin{enumerate}
\item For every $p \in \bbP$ and  complete set $x \in [\ka]^\om$,
if $x \supseteq \dom(p)$ then there is $q \in \bbP$ such that
$q\le p$ and $\dom(q)=x$.
\item For every $x \in [\ka]^\om$,
the set $\{p \in \bbP\mid  x \subseteq \dom(p)\}$ is dense in $\bbP$.
\end{enumerate}
\end{lemma}
\begin{proof}
(1) Take $p \in \bbP$. 
Let $\seq{\alpha_n\mid n<\om}$ be an enumeration of $x \setminus \dom(p)$.
Note that $\calE^{\alpha_n} \cap \dom(p)$ is finite for every $n<\om$.
Thus we can take a function $f$ on $\{\alpha_n\mid n<\om\}$ such that
$f(\alpha_n) \in F(\alpha_n) \setminus (p``(\calE^{\alpha_n} \cap \dom(p)) \cup f``\{\alpha_m\mid m<n\})$.
Let $q=p \cup f$.
It is easy to check that $q \in \bbP$ and $q \le p$.

(2) follows from (1) and Lemma \ref{2.3+}.
\end{proof}

\begin{lemma}\label{2.6}
\begin{enumerate}
\item For $p, q\in \bbP$, 
if
$p \cup q$ is a function then $p \cup q$ is the greatest lower bound of $p$ and $q$.
\item $\bbP$ satisfies the $\om_2$-c.c.
\end{enumerate}
\end{lemma}
\begin{proof}
(1).  
Suppose $r=p \cup q$ is a function.
The set $\dom(r)=\dom(p) \cup \dom(q)$ is complete and $r(\alpha) \in F(\alpha)$ for every
$\alpha \in \dom(r)$ since $p$ and $q$ are conditions.
Thus it is enough to check that $r=p \cup q$ is a good coloring of $\dom(p) \cup \dom(q)$.
Take $\alpha, \beta \in \dom(p) \cup \dom(q)$ with $\alpha \relE \beta$.
We may assume $\alpha<\beta$.
If $\beta \in \dom(p)$, then $\alpha \in \calE^\beta \cap \beta \subseteq \dom(p)$.
Hence $r(\alpha)=p(\alpha) \neq p(\beta)=r(\beta)$.
The case $\beta \in \dom(q)$ follows from the same argument.

(2). For a given $\{p_i\mid i<\om_2\} \subseteq \bbP$,
by the $\Delta$-system lemma,
there are $D \in [\om_2]^{\om_2}$ and
$d$ such that $\dom(p_i) \cap \dom(p_j)=d$ for every distinct $i, j \in D$.
For $\alpha \in d$ and $i \in D$,
we have $p_i(\alpha) \in F(\alpha) \in [\ka]^\om$.
Thus, by a standard  pigeonhole argument,
there is $D' \in [D]^{\om_2}$ such that
$p_i \restriction d=p_j \restriction d$ for every $i, j \in D'$.
Then for every $i, j \in D'$, $p_i \cup p_j$ is a common extension of $p_i$ and $p_j$
by (1).
\end{proof}

The following lemmas are straightforward.
\begin{lemma}\label{2.7}
For a descending sequence $\seq{p_n\mid n<\om}$ in $\bbP$,
if the set $\bigcup_{n<\om} \dom(p_n)$ is complete then
$\bigcup_{n<\om} p_n \in \bbP$.
In particular, for every countable $M \prec \calH_\theta$
%\sup(M \cap \om_2) \notin S^*$,
and every $(M, \bbP)$-generic sequence $\seq{p_n\mid n<\om}$,
if $M \cap \ka$ is complete then
the union $\bigcup_{n<\om} p_n$ is a strong $(M, \bbP)$-generic condition with
domain $M \cap \ka$.
Hence $\bbP$ is $\sigma$-Baire.
\end{lemma}

\begin{lemma}\label{2.11}
\begin{enumerate}
\item $\bbP$ preserves all cofinalities.
\item Let $G$ be a $(V, \bbP)$-generic filter.
Then $f=\bigcup G$ is a good coloring of $X$ with
$f(\alpha) \in F(\alpha)$ for every $\alpha<\ka$.
\end{enumerate}
%Moreover, if $S^*$ is stationary in $V$ then it remains stationary in $V[G]$.
\end{lemma}

We do not know if the poset $\bbP$ is proper or even semiproper.
\begin{question}
Is $\bbP$ proper or semiproper?
\end{question}

Next let us consider complete suborders of $\bbP$.

\begin{define}
For a subset $Y \subseteq \ka$,
let $\bbP \restr Y=\{p \in \bbP \mid \dom(p) \subseteq Y\}$.
We identify $\bbP \restr Y$ as a suborder of $\bbP$.
\end{define}

\begin{lemma}\label{2.8}
Let $Y \subseteq \ka$ be a complete set of size $\om_1$.
\begin{enumerate}
\item The poset $\bbP \restr Y$ is a complete suborder of $\bbP$.
Moreover, for each $p \in \bbP$, the function $p\restr Y$ is in $\bbP \restr Y$ and is the greatest reduction of $p$.
\item Let $M \prec \calH_\theta$ be a countable model with $\bbP, Y \in M$.
For every $(M, \bbP \restr Y)$-generic sequence $\seq{p_n\mid n<\om}$,
the union $\bigcup_{n<\om} p_n$ is a strong $(M, \bbP \restr Y)$-generic condition.
% \item In $V^{\bbP \cap M^*}$,
% the quotient poset $\bbP/(\bbP \cap M^*)$ is $\om_1$-stationary preserving.
\end{enumerate}
\end{lemma}
\begin{proof}
(1). For $p, q \in  \bbP \restr Y$,
if $p$ is compatible with $q$ in $\bbP$,
then $p \cup q$ is a common extension of $p$ and $q$, and $\dom(p \cup q) \subseteq Y$.
Hence $p \cup q \in \bbP \restr Y$,
and $p$ is compatible with $q$ in $\bbP \restr Y$.

Next take $p \in \bbP$. 
The sets $\dom(p)$ and $Y$ are complete, hence $\dom(p) \cap Y$ is complete as well.
Because $\dom(p \restr Y)=\dom(p) \cap Y$, we have $p \restr Y \in \bbP \restr Y$.
We show $p \restr Y$ is a reduction of $p$, that is, for every $r \in \bbP \restr Y$,
if $r \le  p \restr Y$ then $r$ is compatible with $p$ in $\bbP$,
and this is immediate from Lemma \ref{2.6}, and it is straightforward to check that
$p \rest Y$ is the greatest reduction of $p$.

(2). By Lemma \ref{2.3+}, $M \cap Y$ is complete.
Since $\dom(\bigcup_{n<\om} p_n)=M \cap Y$,
the union $\bigcup_{n<\om} p_n$  is a condition in $\bbP \restr Y$.
\end{proof}

\begin{lemma}
Let $N \prec \calH_\theta$ be a $\sigma$-closed model with $\size{N}=\om_1$ and $\bbP \in N$.
Then $\bbP \rest (N \cap \ka)=\bbP \cap N$ and $\bbP \cap N$ is a complete suborder of $\bbP$.
\end{lemma}
\begin{proof}
The inclusion $\bbP \cap N \subseteq \bbP \rest (N \cap \ka)$ is easy.
For the converse,
let $p \in \bbP \rest (N \cap \ka)$.
We know $\dom(p) \subseteq N \cap \ka$. 
Because $p \subseteq \dom(p) \times \{F(\alpha) \mid \alpha \in \dom(p)\} \subseteq N$,
we have $p \in N$ by the $\sigma$-closure of $N$.

The set $N \cap \ka$ is complete by Lemma \ref{2.3}, 
hence $\bbP \rest (N \cap \ka)=\bbP \cap N$ is a complete suborder of $\bbP$.
\end{proof}

Let us say that a poset $\bbQ$ is 
\emph{$\om_1$-diamond preserving} 
if for every stationary $S\subseteq \omega_1$ and
$\diamondsuit(S)$-sequence $\seq{d_\alpha \mid \alpha \in S}$,
$\bbQ$ forces 
``$\seq{d_\alpha \mid \alpha \in S}$ remains a $\diamondsuit(S)$-sequence''.
\begin{lemma}\label{5.21}
\begin{enumerate}
\item $\bbP$ is $\om_1$-stationary preserving
and $\om_1$-diamond preserving.
\item For every complete set $Y \subseteq \ka$,
the quotient $\bbP/(\bbP \restr Y)$ is $\om_1$-stationary preserving.
\end{enumerate}
\end{lemma}
\begin{proof}
First we note the following:
Let $A \in V^{\bbP}$ be a subset of $\om_1$,
and $\dot A$ be a name for $A$.
Since $\bbP$ has the $\om_2$-c.c.,
we can take a complete set $Y \subseteq \ka$ with size $\om_1$
such that $\dot A$ is a $\bbP \restr Y$-name.
Thus, in order to show that $\bbP$ satisfies (1),
it is enough to prove that 
for every complete set $Y \subseteq \ka$ with size $\om_1$,
the complete suborder $\bbP \restr Y$ satisfies (1).
We only show that $\bbP \rest Y$ is $\om_1$-diamond preserving,
the $\om_1$-stationary preservingness follows from the same argument.

Fix a stationary $S \subseteq \om_1$ and a $\diamondsuit(S)$-sequence $\seq{d_\alpha \mid \alpha \in S}$.
Take a $\bbP \rest Y$-name $\dot A$ for a subset of $\om_1$.
Then take an internally approachable sequence $\seq{M_i \mid i<\om_1}$ of countable elementary submodels of $\calH_\theta$
containing all relevant objects,
that is, $\seq{M_i \mid i \le j} \in M_{j+1}$ for every $j<\om_1$,
and $M_j=\bigcup_{i<j} M_i$ if $j$ is limit.
By Lemma \ref{2.8}, we can construct a descending sequence $\seq{p_i \mid i<\om_1}$ in $\bbP$
such that each $p_i$ is a strong $(M_i, \bbP \rest Y)$-generic condition,
and $\seq{p_i \mid i \le j} \in M_{j+1}$.
Let $A=\{\gamma<\om_1 \mid \exists i<\om_1\,(p_i \Vdash_{\bbP \rest Y} \gamma \in \dot A)\}$.
Since $\seq{d_\alpha \mid \alpha \in S}$ is a $\diamondsuit(S)$-sequence,
there is some $\alpha \in S$ such that
$d_\alpha=A \cap \alpha$ and $M_\alpha \cap \om_1=\alpha$.
We show that
$p_\alpha \Vdash_{\bbP\rest Y}$``$d_\alpha=\dot A \cap \alpha$''.
Let $\gamma \in d_\alpha$. Since $A \cap \alpha=d_\alpha$,
there is some $i<\om_1$ such that $p_i \Vdash_{\bbP \rest Y} \gamma \in \dot A$.
Pick $j<\alpha$ with $\gamma \in M_j$. Note that $p_\alpha \le p_j$.
Since $p_j$ is a strong $(M_j, \bbP\restriction Y)$-generic condition,
we know $p_j \Vdash_{\bbP \rest Y} \gamma \in \dot A$,  or 
$p_j \Vdash_{\bbP \rest Y} \gamma \notin \dot A$.
$p_i$ is compatible with $p_j$, so we have $p_j \Vdash_{\bbP \rest Y} \gamma \in \dot A$,
and $p_\alpha \Vdash_{\bbP \rest Y} \gamma \in \dot A$.
This means that $p_\alpha \Vdash_{\bbP \rest Y}$``$d_\alpha \subseteq  \dot A \cap \alpha$''.
For the converse, let $\gamma<\alpha$ and suppose $\gamma \notin d_\alpha$.
Pick $j<\alpha$ with $\gamma \in M_j$. 
Since $\gamma \notin d_\alpha=A \cap \alpha$,
$p_j$ does not force $\gamma \in \dot A$.
Again, since $p_j$ is strong $(M_j, \bbP\rest Y)$-generic,
we have that  $p_j$ forces $\gamma \notin \dot A$,
and $p_\alpha \Vdash_{\bbP \rest Y} \gamma \notin \dot A$.
This means that $p_\alpha \Vdash_{\bbP \rest Y}$``$\dot A \cap \alpha \subseteq d_\alpha$''.  

For (2), take another complete set $Z \subseteq \ka$ with size $\om_1$ and $Y \subseteq Z$.
$\bbP \restr Y$ is a complete suborder of $\bbP \restr Z$.
By the same reason as before, it is enough to show that
$(\bbP \rest Z)/(\bbP \rest Y)$ is $\om_1$-stationary preserving.

Take a $(V, \bbP \rest Y)$-generic $G$, and we work in $V[G]$.
Fix a stationary set $S \subseteq \om_1$.
Take a countable  elementary submodel $M' \prec \calH_\theta^{V[G]}$
such that $M' \cap \om_1 \in S$,
and an $(M', (\bbP\rest Z)/G)$-generic sequence $\seq{p_n \mid n<\om}$.
It is enough to prove that this sequence has a lower bound in $(\bbP \rest Z)/G$.

Here note that $p_n \rest Y \in G$ for every $n<\om$.
Let $M=M' \cap \calH_\theta^V$, which is an elementary submodel of $\calH_\theta^V$ in $V$,
and $\seq{p_n \mid n<\om}$ is an $(M, \bbP\rest Z)$-generic sequence belongs to $V$.
By Lemma \ref{2.8} again, the union $p=\bigcup_n p_n$ is in $\bbP \rest Z$.
Moreover $p\rest Y=\bigcup_n (p_n \rest Y)$, which belongs to $G$.
Hence $p \in (\bbP \rest Z)/G$, as required.
\end{proof}

If $X$ is a trivial graph on $\ka$ (that is, $X=\seq{\ka, \emptyset})$ and $F(\alpha)=\om$,
then it is clear that $\bbP(X, F)$ is isomorphic to
$\mathrm{Fn}(\ka, \om, <\om_1)$.
It is known that 
for every stationary $S \subseteq \om_1$,
$\mathrm{Fn}(\ka, \om, <\om_1)$ forces $\diamondsuit(S)$.
Thus we have:

\begin{lemma}\label{2.10}
If $X$ is a trivial graph and
$F(\alpha)=\om$ for $\alpha<\ka$, then
for every stationary $S \subseteq \om_1$,
$\bbP(X, F)$ forces $\diamondsuit(S)$.
\end{lemma}

Now we consider an iteration, this is the hard part of this section.
Let $l$ be an ordinal and $\seq{\bbP_\xi, \dot \bbQ_\eta\mid \eta \le \xi <l}$ a countable support iteration
satisfying the following induction hypotheses:
\begin{enumerate}
\item $\bbP_\xi$ satisfies the $\om_2$-c.c. and is $\sigma$-Baire.
\item $\bbP_\xi$ is $\om_1$-stationary preserving and $\om_1$-diamond preserving.
%\item $\bbP_\xi$ forces that ``Conditions  (1) and (2) in Assumption 1 hold''.
\item For $\xi<l$, there are $\bbP_\xi$-names $\dot X_\xi$ and $\dot F_\xi$ such that
$\Vdash_{\bbP_\xi}$``$\dot X_\xi$ is a $\ka$-nice graph,
$\dot F_\xi:\ka \to [\ka]^\om$, and $\dot \bbQ_\xi=\bbP(\dot X_\xi, \dot F_\xi)$''
\item For $\xi<l$, let $D_\xi$ be the set of all $p \in \bbP_\xi$
such that for every $\eta \in \supp(p)$ there is $r_\eta$ such that $p(\eta)$ is the canonical name of $r_\eta$.
Then $D_\xi$ is dense in $\bbP_\xi$.
\end{enumerate}
Let $\xi<l$ and $N \prec \calH_\theta$ be a $\sigma$-closed model
containing all relevant objects and $\size{N}=\om_1$.
Let $p \in D_\xi$. For $\eta \in \supp(p)$, let $r_\eta$ be the function such that
$p(\eta)$ is the canonical name of $r_\eta$.
Let $p^{N}$ be the function defined by $\dom(p^{N})=\xi$,
$\supp(p^{N})=\supp(p) \cap N$,
and for $\eta \in \supp(p^{N})$,
$p^{N}(\eta)$ is the canonical name of $r_\eta \rest (N \cap \ka)$.
%Note also that $p^N \in N$ since $N$ is $\sigma$-closed.

Let $\seq{p_n\mid n<\om}$ be a descending sequence in $D_\xi$.
For $n<\om$ and $\eta \in \supp(p_n)$,
let $r_{n,\eta}$ be the function such that
$p_n(\eta)$ is the canonical name of $r_{n,\eta}$.
The \emph{canonical limit} of $\seq{p_n\mid n<\om}$ is 
the function $p$ defined by $\dom(p)=\xi$, 
$\supp(p)=\bigcup_{n<\om} \supp(p_n)$,
and for $\eta \in \supp(p)$, $p(\eta)$ is the canonical name
of $\bigcup\{r_{n,\eta}\mid n<\om, \eta \in \supp(p_n)\}$.
Notice that the canonical limit is not necessarily a condition.

Now we also require the following for the induction hypotheses:
\begin{enumerate}
\item[(5)] Let $\xi<l$ and $N \prec \calH_\theta$ be a $\sigma$-closed model
containing all relevant objects and $\size{N}=\om_1$.
\begin{enumerate}
\item$\bbP_\xi \cap N$ is a complete suborder of $\bbP_\xi$.
In addition, for $p \in D_\xi$, $p^N \in \bbP_\xi \cap N$ and is the greatest reduction of $p$.
\item The quotient $\bbP_\xi/(\bbP_\xi \cap N)$ is $\om_1$-stationary preserving.
%(In particular, for a $(V, \bbP_\xi/(\bbP_\xi \cap M^*))$-generic $G^*$,
%$p \in \bbP/(\bbP_\xi \cap M^*)$ if and only 
%if $p^{M^*} \in G^*$)
%\item Let $\eta+1 \le \xi$.
%Let $M \prec H_\theta$ be countable with $\bbP_\eta, N,\dotsc \in M$.
%Then for every strong $(M, \bbP_\eta \cap N)$-generic condition $p$,
%$p\Vdash_{\bbP_\eta}$``$M \cap N \cap \ka$ is $\seq{\dot X_\eta, \dot F_\eta}$-complete''.
\item Let $M \prec \calH_\theta$ be countable with $\bbP_\xi, N,\dotsc \in M$.
Then every $(M, \bbP_\xi \cap N)$-generic sequence in $D_\xi \cap N$ has a lower bound.
More precisely, let $\seq{p_n\mid n<\om}$ be an $(M, \bbP_\xi \cap N)$-generic sequence with $p_n \in D_\xi \cap N$.
Then the canonical limit is a condition in $\bbP_\xi \cap N$,
hence is a lower bound of $\seq{p_n \mid n<\om}$.
% there is a strong $(M, \bbP_\xi \cap M^*)$-generic condition $q \le p$
% such that $\supp(q)=N \cap \xi$ and $\dom(r_\eta)=N \cap \om_2$,
% where $q(\eta)=\check r_\eta$ for $\eta \in \supp(q)$.
% In addition $q$ is a strong $(N, \bbP_\xi)$-generic condition.
\end{enumerate}
% \item Let $\theta$ be a sufficiently large regular $\theta$,  and $\Delta$ a well-ordering on $H_\theta$.
% Let $M \prec \seq{H_\theta,\in \Delta}$ be a countable model with
% $\sup(M \cap \om_2) \notin S^*$ and contains all relevant objects.
% Then for every $(M, \bbP_l)$-generic sequence $\seq{p_n\mid n<\om}$,
% there is $p \in D_\xi$ which is strong $(M, \bbP_l)$-generic, a lower bound of the $p_n$'s, and
% $\dom(r_\eta)=M \cap \om_2$ for every $\eta \in \supp(p)$ (where $p(\eta)=\check r_\eta$).
\end{enumerate}

We define $\bbP_l$ as intended.
Now we verify that $\bbP_l$ satisfies the induction hypotheses.
\\

Case 1: $l$ is successor, say $l=k+1$.
(1)--(4) follow from the induction hypotheses and lemmas above.

(5). Fix a $\sigma$-closed  $N \prec \calH_\theta$ containing all relevant objects and $\size{N}=\om_1$.
As usual, we can identify $\bbP_l$ with $\bbP_k *\dot \bbQ_k$.
By (4), $D_l$ is dense in $\bbP_l$.

Before rendering the proof of (5), we give some observations:
Let $G$ be $(V, \bbP_k)$-generic, and
$G^*=G \cap N$ which is  $(V, \bbP_k \cap N)$-generic. 
Let $X_k$ and $F_k$ be the interpretations of $\dot X_k$ and $\dot F_k$ by $G$ respectively.
We know $N[G] \prec \calH_\theta^{V[G]}$. Since $\bbP_k$ satisfies the $\om_2$-c.c.\ and $N$ is $\sigma$-closed,
we have that $N[G]$ remains $\sigma$-closed and
$N[G] \cap \mathrm{ON}=N \cap \mathrm{ON}$.
Hence $N \cap \ka$ is $\seq{X_k, F_k}$-complete by Lemma \ref{2.3}.
%Let $X_k=\seq{\ka,\calE_k}$ be the interpretation of
%$\dot X_k$ by $G$, and $F_k$ be of $\dot F_k$.
%Let $X^*=X \cap M^*$ be the induce subgraph of $X$.
For each $\alpha, \beta \in Y=N \cap \ka$,
there is a maximal antichain $I \in N$ in $\bbP_k$ which decides whether $\alpha \relE_k \beta$ or not.
$I \subseteq \bbP_k \cap N$, hence $I$ is a maximal antichain in $\bbP_k \cap N$.
Thus  we have that the induced subgraph $Y$ of $X_k$ lies in $V[G^*]$.
Similarly, we have $F_k \restriction Y \in V[G^*]$ and  $F_k(\alpha), \calE_k^\alpha \cap \alpha \subseteq Y$ for $\alpha 
\in Y$.
Moreover, for each countable $y \subseteq Y$, we know $y \in N[G]$, and
the set $\{\gamma<\ka\mid \size{\calE_k^\gamma \cap y} \ge \om\}$ is in $N[G]$.
Hence it is a subset of $N[G] \cap \ka=N\cap \ka=Y$ by Lemma \ref{2.1}.
Thus, for each $y \in [Y]^\om$,
$y$ is $(X_k, F_k)$-complete in $V[G]$ if and only if $y$ is $(Y, F_k \rest Y)$-complete in $V[G^*]$, i.e.,
$\calE^\alpha \cap \alpha, F_k(\alpha) \subseteq y$ for $\alpha \in y$, and
for every $\beta \in Y$, if $\calE_k^\beta \cap y$ is infinite then $\beta \in y$.
%Consequently, $\{ p \in \bbQ_k \mid \dom(p) \subseteq Y \}$ is definable in $V[G^*]$,
%namely, it is the set of all partial countable coloring on $Y$
%with $(Y, F_k \rest Y)$-complete domain. Hence it lies in $V[G^*]$.
In addition, since $\col(Y) \le \om$ in $V[G]$ and
$\bbP_k/(\bbP_k \cap N)$ is $\om_1$-stationary preserving,
we have that $\col(Y) \le \om$ in $V[G^*]$ by
Lemma \ref{1.11}.

(a).
Let $p_0, p_1 \in \bbP_l \cap N$,
and suppose $p_0$ is incompatible with $p_1$ in $\bbP_l \cap N$.
We may assume that $p_0, p_1 \in D_l$,
so, for $i<2$, $p_i$ is of the form $\seq{q_i, r_i}$.

If $q_0$ is incompatible with $q_1$ in $\bbP_k \cap N$,
then these are incompatible in $\bbP_k$ by the induction hypotheses,
and so $\seq{q_0,  r_0}$ and $\seq{q_1, r_1}$ are incompatible in $\bbP_l$.
Suppose $q_0$ and $q_1$ are compatible. Let $q \in \bbP_k \cap N$ be a common extension.
Now, if $r_0 \cup r_1$ is a function,
then  we have that $\seq{q, (r_0 \cup r_1)} \in \bbP_l \cap N$ and is a common extension
of $\seq{q_0, r_0}$ and $\seq{q_1, r_1}$.
Thus  $\seq{q_0, r_0}$ and $\seq{q_1, r_1}$ are compatible in $\bbP_l \cap N$,
this is a contradiction.
Hence $r_0 \cup r_1$ is not a function, and we have that
$p_0$ and $p_1$ are incompatible in $\bbP_l$.

Next take  $p=\seq{q, r} \in D_l$.
We shall prove that $p^N \in N$ and is the greatest reduction of $p$.
Now,  we identify $p^{N}$ with $\seq{q^{N},  r\restriction (N \cap \ka)}$.
We have $q^N \in N$ by the induction hypotheses,
and $\dom(r \restr (N \cap \ka)) \in N$ by the $\sigma$-closure of $N$.
Since $\bbP_k$ satisfies the $\om_2$-c.c.,
we can find a set $A \subseteq \ka$ such that $A \in N$, 
$\size{A} \le \om_1$,
and $\Vdash_{\bbP_k}$``$\dot F_k``(\dom(r \restr (N \cap \ka))) \subseteq A$''.
We know $A \subseteq N$, and 
$r \rest (N \cap \ka) \subseteq (\dom(r) \cap N \cap \ka) \times A$
because $q \Vdash_{\bbP_k} r \in \dot \bbQ_k$.
Thus we have $r \rest (N \cap \ka) \in N$.

Next we show that $q^{N} \force_{\bbP_k}$``$ r \restriction (N \cap \ka) \in \dot \bbQ_k$''.
For this, first we show that
$q^{N} \force_{\bbP_k}$``$\dom(r \restriction (N \cap \ka))$ is $\seq{\dot X_k, \dot F_k}$-complete''.
If not, by the elementarity of $N$,
there is $q ' \le q^N$ such that $q' \in \bbP_k \cap N$
and
$q' \force_{\bbP_k}$``$\dom( r \restriction (N \cap \ka))$ is not $\seq{\dot X_k, \dot F_k}$-complete''.
$q'$ is compatible with $q$.
On the other hand, we know $q \Vdash_{\bbP_k}$``$\dom(r)$ is complete'' since $q$ forces $r \in \dot \bbQ_k$,
and, $\Vdash_{\bbP_k}$``$N \cap \ka$ is complete'' because of the observation before.
Hence $q \Vdash_{\bbP_k}$``$\dom(r \restriction (N \cap \ka))=\dom(r) \cap N \cap \ka$ is complete'',
this is a contradiction.
The same argument shows that
$q^{N} \force_{\bbP_k}$``$r \restriction (N \cap \ka)$ is a good coloring'',
thus we have
$q^{N} \force_{\bbP_k}$``$r \restriction (N \cap \ka) \in \dot \bbQ_k$''.
Finally we must show that $\seq{q^N, r \restriction (N \cap \ka)}$ is the greatest reduction,
but this can be verified by a standard argument.

(c). Take a countable $M \prec \calH_\theta$ containing all relevant objects.
First we prove the following claim:
\begin{claim}
For every strong $(M, \bbP_k \cap N)$-generic condition $p$,
$p\Vdash_{\bbP_k}$``$M \cap N \cap \ka$ is $\seq{\dot X_k, \dot F_k}$-complete''.
\end{claim}
\begin{proof}[Proof of Claim]
Let $p$ be a strong $(M, \bbP_k \cap N)$-generic condition.
Take a $(V, \bbP_k)$-generic $G$ with $p \in G$,
and let $G^*$ be $(V, \bbP_k\cap N)$-generic induced by $G$.

Let $X_k=\seq{\ka, \calE_k}$ and $F_k$ be the interpretations of $\dot X_k$ and $\dot F_k$ by $G$ respectively.
We know that $M[G^*] \cap \mathrm{ON}=M \cap \mathrm{ON}$
(but $M[G] \cap \mathrm{ON} \neq M \cap \mathrm{ON}$ may be possible).
Let $Y=N \cap \ka$.
Now we show that $M \cap Y$ is $\seq{X_k, F_k}$-complete in $V[G]$.
By the observation before,
it is enough to show that $M \cap Y$ is $\seq{Y, F_k \rest Y}$-complete in $V[G^*]$.
We argue as in the proof of Lemma \ref{2.3+}. We work in $V[G^*]$.
We know that $N[G] \cap \ka=N[G^*] \cap \ka=N \cap \ka=Y$ and $Y$ is $(Y, F_k \rest Y)$-complete in $V[G^*]$.
We also know $\col(Y) \le \om$ in $V[G^*]$. In addition we have $Y, F_k \rest Y \in M[G^*]$.
Since $M[G^*] \prec \calH_\theta^{V[G^*]}$ and $M \cap \ka=M[G^*] \cap \ka$,
it is enough to check that 
for every $\beta \in Y$, if $M[G^*]\cap Y \cap \calE^\beta_k$ is infinite
then $\beta \in M[G^*] \cap Y$.
%$Y$ is complete, so we have $\beta \in Y$.
Since $\col(Y) \le \om$,
there is $f: Y \to [Y]^{<\om}$ in $M[G^*]$ such that
for every $\alpha, \alpha' \in Y$, if $\alpha \mathrel\calE_k \alpha'$ then
$\alpha \in f(\alpha')$ or $\alpha' \in f(\alpha)$.
$M[G^*] \cap Y \cap \calE^\beta_k$ is infinite but $f(\beta)$ is finite,
so there is $\alpha \in (M[G^*] \cap Y \cap \calE^\beta_k) \setminus f(\beta)$.
Then $\beta \in f(\alpha) \subseteq M[G^*]$.
%Then the same argument used in the proof of Lemma \ref{2.3+}
%shows that $M \cap Y$ is $\seq{Y, F_k \rest Y}$-complete. 
\end{proof}

Now we prove (c). Take an $(M, \bbP_l \cap N)$-generic sequence $\seq{p_n \mid n<\om}$ in $D_l \cap N$.
We can identify $p_n$ as $\seq{q_n,  r_n}$ where $q_n \in D_k \cap N$, $r_n \in N$,
and $q_n \Vdash_{\bbP_k} r_n \in \dot \bbQ_k$.
The sequence $\seq{q_n \mid n<\om}$ is an $(M, \bbP_k \cap N)$-generic sequence.
By the induction hypotheses, the canonical limit $q$ of this sequence is a condition in $\bbP_k \cap N$.
%We know that $q \Vdash_{\bbP_k \cap N}$``$N[\dot G^*] \in M[\dot G^*]$,
%$N[\dot G^*] \cap \ka=N \cap \ka$, $M[\dot G^*] \cap \ka=M \cap \ka$,
%and $\seq{\check r_n \mid n<\om}$ is a $(M[\dot G^*], \dot \bbQ \cap N[G^*])$-generic sequence''.
By the claim above, we know $q \Vdash_{\bbP_k}$ ``$M \cap N \cap \ka$ is $(\dot X_k, \dot F_k)$-complete''.
Since $\dom(\bigcup_n r_n)=M \cap N \cap \ka$, one can check that 
$q \Vdash_{\bbP_k}$``$\bigcup_n r_n \in \bbQ_k$'',
thus $\seq{q, (\bigcup_n r_n)}$ is a lower bound
of $\seq{p_n \mid n<\om}$ in $\bbP_l \cap N$,
and  is the canonical limit of $\seq{p_n \mid n<\om}$.

(b). Take another $\sigma$-closed model $N' \prec \calH_\theta$
with $\size{N'}=\om_1$ and $N,\dotsc \in N'$.
As in the proof of Lemma \ref{5.21}, it is sufficient to prove
that $(\bbP_l \cap N')/(\bbP_l \cap N)$ is $\om_1$-stationary preserving.

Take a $(V, \bbP_l \cap N)$-generic $G^*$ and work in $V[G^*]$.
Fix a stationary set $S \subseteq \om_1$.
Take a countable model $M' \prec \calH_\theta^{V[G^*]}$
such that $M'$ contains all relevant objects and $M' \cap \om_1 \in S$.
It is enough to prove that there is an $(M', (\bbP_l \cap N')/G^*)$-generic condition.
Take an $(M', (\bbP_l \cap N')/G^*)$-generic sequence $\seq{p_n \mid n<\om}$ with $p_n \in D_l \cap N$.
As before, we may assume that each $p_n$ is of the form $\seq{q_n, r_n}$.
Note that $\seq{p_n \mid n<\om} \in V$.
Let $M=M' \cap V$, which is a countable elementary submodel of $\calH_\theta^V$ and $M \in V$.
We know $M[G^*]=M'$.
In addition, $\seq{p_n \mid n<\om}$ is an $(M, \bbP_l \cap N')$-generic sequence,
and $\seq{q_n \mid n<\om}$ is $(M, \bbP_k \cap N')$-generic.
By the induction hypotheses, the canonical limit $q$ of the sequence $\seq{q_n \mid n<\om}$
is a condition of $\bbP_k \cap N'$.
Let $r=\bigcup_n r_n \in N'$.
By (c) and the claim above, $q \Vdash_{\bbP_k}$``$\dom(r)=M \cap N' \cap \ka$ is $(\dot X_k, \dot F_k)$-complete'',
hence $q \Vdash_{\bbP_k}$``$r \in \dot \bbQ_k$'',
and $\seq{q, r} \in \bbP_l \cap N'$.
Now consider the reduction $q_n^N$ of $q_n$. We have $q_n^N \in G^* \cap \bbP_k$.
By the construction of $q$, the reduction $q^N$ of $q$ is also in $G^* \cap \bbP_k$.
In addition, $q^N$ is a strong $(M, \bbP_k \cap N)$-generic condition.
By the claim, we have $q^N \Vdash_{\bbP_k}$``$\dom(r^N)=M \cap N \cap \ka$ is $(\dot X_k, \dot F_k)$-complete'',
hence $q^N \Vdash_{\bbP_k}$``$r^N \in \dot \bbQ_k$'',
and $\seq{q,r}^N=\seq{q^N, r^N} \in G^*$.
Combining these arguments,
we have that $\seq{q,  r} \in (\bbP_l \cap N')/G^*$.
\\

Case 2: $l$ is limit. (3) is trivial.
For (1), the chain condition of $\bbP_l$ follows from (4) and the standard $\Delta$-system argument.
The $\sigma$-Baireness follows from (5).
We show that (2), (4), and (5) hold.

(4). Take $p \in \bbP_l$.
Now fix a $\sigma$-closed model $N \prec \calH_\theta$ of size $\om_1$ and $p, \bbP_\xi,\dotsc \in N$.
Here notice that we do not know yet that $\bbP_l \cap N$ is a complete  suborder of $\bbP_l$,
but this does not cause a problem.
Take a countable $M \prec \calH_\theta$ with $N, p,\dotsc \in M$.
Take an $(M, \bbP_l \cap N)$-generic sequence $\seq{p_n\mid n<\om}$ with
$p_0 \le p$.
Note that $\supp(p_n) \subseteq M \cap N \cap l$ for every $n<\om$. 
Fix an increasing sequence $\seq{l_m\mid m<\om}$ with limit $\sup(M \cap N \cap l)$ and $l_m \in M \cap N \cap l$.
For each $m<\om$,
the sequence $\seq{p_n \restriction l_m \mid n<\om}$ is an $(M, \bbP_{l_m} \cap N)$-generic sequence.
By the induction hypotheses,
the canonical limit $q_m$ of $\seq{p_n \restriction l_m \mid n<\om}$ is a condition and
in $D_{l_m} \cap N$.
Clearly we have $q_m=q_n \restriction l_m$ for $m<n<\om$.
Let $q=\bigcup_{m<\om} q_m$. We have that $ q\in D_l$ and $q \le p$.

(5). 
Fix a $\sigma$-closed $N \prec \calH_\theta$ of size $\om_1$ containing all relevant objects.

For (a),  let $p \in \bbP_l$. Then for each $\xi \in \supp(p) \cap N$,
$(p \restriction \xi)^{N} \in N$ is a reduction of $p \rest \xi$.
Then clearly $p^{N}=\bigcup_{\xi \in \supp(p) \cap N} (p \restriction \xi)^{N} \in N$ and
is the greatest reduction of $p$.

For (c), take a countable $M \prec \calH_\theta$ with $N,\dotsc \in M$.
Fix an $(M, \bbP_l \cap N)$-generic sequence $\seq{p_n\mid n<\om}$.
Take an increasing sequence $\seq{l_m\mid m<\om}$ with limit $\sup(N \cap M \cap l)$ and
$l_m \in N \cap M \cap l$.
For $m<\om$, the sequence $\seq{p_n \restriction l_m\mid n<\om}$ is an $(M, \bbP_{l_m} \cap N)$-generic sequence,
so the canonical limit $q_m$ of the sequence is a condition in $\bbP_{l_m} \cap N$ by the induction hypotheses.
Then $q=\bigcup_{m<\om} q_m$ is the canonical limit of $\seq{p_n\mid n<\om}$ and
a condition in $\bbP_l \cap N$.

Finally we have to verify the conditions (5)(b) and (2).
This can be done by the same argument used in the proof of Lemma \ref{5.21}
with the condition (5)(c),
so we show only (2).
For the $\om_1$-stationary preserving property of $\bbP_l$,
fix a stationary set $S \subseteq \om_1$, $p \in \bbP_l$, and a $\bbP_l$-name $\dot C$ for a club in $\om_1$.
Take a $\sigma$-closed $N \prec \calH_\theta$ with size $\om_1$ and $S, \dot C, \dotsc \in N$.
Since $\bbP_l$ satisfies the $\om_2$-c.c., 
we may assume that $\dot C$ is a $\bbP_l \cap N$-name.
Take a countable model $M \prec \calH_\theta$ containing all relevant objects such that
$M \cap \om_1 \in S$.
By (5)(c), we can take a strong $(M, \bbP_l \cap N)$-generic condition $q \le p$.
We have $q \Vdash_{\bbP_l \cap N} M \cap \om_1 \in \dot C$,
and since $\bbP_l \cap N$ is a complete suborder,
we also 
have $q \Vdash_{\bbP_l} M \cap \om_1 \in \dot C$.
For the $\om_1$-diamond preserving property, 
let $\seq{d_\alpha \mid \alpha \in S}$ be a $\diamondsuit(S)$-sequence,
and $\dot A$ a $\bbP_l$-name for a subset of $\om_1$.
Take a $\sigma$-closed $N \prec \calH_\theta$ of size $\om_1$ and $\seq{d_\alpha \mid \alpha \in S},\dot A,\dotsc \in N$.
As before, we may assume that $\dot A$ is a $\bbP_l \cap N$-name.
Take an internally approachable sequence $\seq{M_i \mid i<\om_1}$ of countable elementary submodels of $\calH_\theta$
containing all relevant objects.
By (5)(c), we can take a descending sequence $\seq{p_i \mid i<\om_1}$ in $\bbP_l \cap N$
such that for each $i<\om_1$, $p_i$ is a strong $(M_i, \bbP_l \cap N)$-generic condition,
and $\seq{p_i \mid i \le j} \in M_{j+1}$.
Let $A=\{\gamma<\om_1 \mid \exists i<\om_1 \, (p_i \Vdash_{\bbP \cap N} \gamma \in \dot A)\}$,
and take $\alpha \in S$ such that $d_\alpha=A \cap \alpha$.
Then, as in Lemma \ref{5.21},
we can show $p_\alpha \Vdash_{\bbP_l \cap N}$``$d_\alpha=\dot A \cap \alpha$'',
and so 
$p_\alpha \Vdash_{\bbP_l}$``$d_\alpha=\dot A \cap \alpha$''.

This completes the proof that $\bbP_l$ satisfies the induction hypotheses (1)--(5).
\\

For $\xi<l$, let $G$ be $(V, \bbP_\xi)$-generic.
In $V[G]$, we can consider the tail poset $\bbP_{\xi,l}=\bbP_l/G$.
Since $\bbP_\xi$ is $\sigma$-Baire and $\om_2$-c.c.,
it is clear that the tail poset $\bbP_{\xi,l}$ is forcing
equivalent to an $(l-\xi)$-stage countable support iteration.
In particular the tail poset $\bbP_{\xi, l}$ is $\om_1$-stationary 
preserving and $\om_1$-diamond preserving.

\begin{lemma}\label{5.20+}
\begin{enumerate}
\item For $\xi<l$, if $Y \in V^{\bbP_\xi}$ is a $\ka$-nice graph,
then
$Y$ remains $\ka$-nice in $V^{\bbP_l}$.
\item Suppose $\cf(l)>\ka$.
Let $X \in V^{\bbP_l}$ be a graph of size $\ka$ and, in $V^{\bbP_l}$,
suppose that there is a $\ka$-nice graph $Y$ which is
isomorphic to $X$.
Then there is some $\xi<l$ such that 
$Y \in \bbP_\xi$ and $Y$ is $\ka$-nice in $V^{\bbP_\xi}$.
\end{enumerate}
\end{lemma}
\begin{proof}
In $V^{\bbP_\xi}$, the tail poset $\bbP_{\xi, l}$ satisfies the $\om_2$-c.c.
and is $\om_1$-stationary preserving.

(1). Let $Y$ be a $\ka$-nice graph in $V^{\bbP_\xi}$.
We check that $Y$ is $\ka$-nice in $V^{\bbP_l}$,
and it is enough to show that
$\col(Y) \le \om_1$, and $\col(Z) \le \om$ for every $Z \subseteq Y$ with size $\le \om_1$ in $V^{\bbP_l}$.

Since $\col(Y) \le \om_1$ in $V^{\bbP_\xi}$,
it is clear that $\col(Y) \le \om_1$ in $V^{\bbP_l}$.
If $Z \in V^{\bbP_l}$ is a subset of $Y$ with size $\om_1$,
by the $\om_2$-c.c. we can find $Z' \in V^{\bbP_\xi}$ such that
$Z \subseteq Z'$ and $\size{Z'} \le \om_1$ in $V^{\bbP_\xi}$.
Since $Y$ is $\ka$-nice in $V^{\bbP_\xi}$,
we have that $\col(Z') \le \om$ in $V^{\bbP_\xi}$.
Then we have that $\col(Z) \le \col(Z') \le \om$ in $V^{\bbP_l}$.

(2). Let $Y$ be a graph in $V^{\bbP_l}$ which is $\ka$-nice and is isomorphic to $X$.
Since $\cf(l)>\ka$ and the chain condition of $\bbP_l$,
there is some $\xi<l$ with $Y \in V^{\bbP_\xi}$.
We have to check that $Y$ is $\ka$-nice in $V^{\bbP_\xi}$,
and, again, it is enough to show that
$\col(Y) \le \om_1$, and $\col(Z) \le \om$ for every $Z \subseteq Y$ with size $\le \om_1$ in $V^{\bbP_\xi}$.
We know $\col(Y) \le \om_1$ in $V^{\bbP_l}$.
By Lemma \ref{1.11} (2), we know $\col(Y) \le \om_1$ in $V^{\bbP_\xi}$.
Next take $Z \in [Y]^{\om_1}$ in $V^{\bbP_\xi}$.
Since $\col(Z) \le \om$ in $V^{\bbP_l}$,
we have that $\col(Z) \le \om$ in $V^{\bbP_l}$
by Lemma \ref{1.11} (1).
\end{proof}

Suppose $2^\ka=\ka^+$
and consider a $\ka^+$-stage iteration $\bbP_{\ka^+}$.
Using the standard book-keeping method and Lemma \ref{5.20+},
we have:
\begin{prop}\label{ite}
Suppose CH, $\ka\ge \om_2$, and $2^{\ka}=\ka^+$.
Then we can construct a poset $\bbP$ which is $\sigma$-Baire,
satisfies the $\om_2$-c.c., 
and forces the following:
\begin{enumerate}
\item $\diamondsuit(S)$ holds for every stationary $S \subseteq \om_1$.
\item For every graph $X$ of size $\le \ka$, 
if $\col(X) \le \om_1$ and $\col(Y) \le \om$ for every $Y \in [X]^{\om_1}$,
then $\List(X) \le \om$.
\end{enumerate}
\end{prop}
%\begin{question}
%Is $\bbP(X,F)$ proper?
%\end{question}

\section{Proofs of Theorem \ref{thm1} and \ref{thm2} (1)}\label{sec6}

In this section, using the forcing notion constructed in Section \ref{sec5},
we give proofs of Theorem \ref{thm1} and \ref{thm2} (1).

First we give the proof of Theorem \ref{thm1}.
Recall that: 
\let\temp\thethm
\renewcommand{\thethm}{\ref{thm1}}
\begin{thm}
Suppose GCH. Let $\la>\om_1$ be a cardinal,
and suppose $\mathrm{AP}_{\mu}$ holds for every $\mu<\la$ with countable cofinality.
Then there is a poset $\bbP$ which is $\sigma$-Baire, satisfies $\om_2$-c.c.,
and forces that ``\,$\Refl(\List, \la)$ holds and $2^{\om_1}>\la$''.
\end{thm}
\begin{proof}[Proof of Theorem \ref{thm1}]
Suppose GCH. 
Let $\ka \ge \om_2$ be a cardinal. 
Suppose $\mathrm{AP}_\la$ holds for every singular $\la<\ka$ of countable cofinality.
By Proposition \ref{ite}, we can take an $\om_2$-c.c., $\sigma$-Baire poset which
forces the following:
\begin{enumerate}
\item $\diamondsuit(S)$ holds for every stationary $S \subseteq \om_1$.
\item $\la^\om=\la$ for every regular uncountable $\la<\ka$.
\item $\mathrm{AP}_\la$ holds for every singular $\la<\ka$ of countable cofinality.
\item For every graph $X$ of size $\le \ka$, 
if $\col(X) \le \om_1$ and $\col(Y) \le \om$ for every $Y \in [X]^{\om_1}$, 
then $\List(X) \le \om$.
\end{enumerate}

Then $V^{\bbP}$ is the required model;
Take a graph $X$ with $\om_2 \le \size{X} \le \ka$ and $\List(X)>\om$.
By (4), we have $\col(X) \ge \om_2$,
or $\col(Y) >\om$ for some $Y \in [X]^{\om_1}$.
If $\col(X) \ge \om_2$, then
there is $Y \in [X]^{\om_1}$ with $\List(Y) >\om$ by (2), (3), and Proposition \ref{1.8}.
If $\col(Y) >\om$ for some $Y \in [X]^{\om_1}$,
then $\List(Y)>\om$ by (1) and Corollary \ref{3.4.1}.
\end{proof}

\let\thethm\temp
\addtocounter{thm}{-1}

%\section{Proof of Theorem \ref{thm2}}
%It is sufficient to construct a model in which:
%\begin{enumerate}
%\item $\diamondsuit(S)$ holds for every stationary $S \subseteq \om_1$.
%\item For every graph $X$ of size $\om_2$, 
%if $\col(Y) \le \om$ for every $Y \in [X]^{<\om_2}$ and  $\col(X) \le \om_1$
%then $\List(X) \le \om$.
%\item $2^\la=\la^+$ for every $\la \ge \om_2$.
%\item $\FRP(\la)$ holds for every regular $\la>\om_2$.
%\item There is a non-reflecting stationary set $S \subseteq \om_2 \cap \cof(\om)$.
%\end{enumerate}
%Why;  Since $\om_2 \cap \cof(\om)$ has a non-reflecting stationary subset,
%we have that $\mathrm{RP}(\col, \om_2)$ fails.
%To prove $\mathrm{RP}(\List)$ holds,
%let $X$ be a graph with $\List(X)>\om$.
%The case $\size{X}=\om_2$ follows from the above argument.
%The case $\size{X} \ge \om_3$ follows from the induction on $\size{X}$.
%Suppose the assertion holds for every graph $Y$ of
%size $<\size{X}$.
%If $\List(X)>\om$ but $\List(Y) \le \om$ for every $Y \in [X]^{\om_1}$,
%then $\List(Z) \le \om$ for every $Z \in [X]^{<\size{X}}$
%by the induction hypothesis.
%Then we have the contradiction by
%Propositions \ref{1.3},
%\ref{1.7}, and (1), (3), and (4).

Before starting the proof of  Theorem \ref{thm2} (1),
we introduce the following notion.
For a poset $\bbP$ and an ordinal $\alpha$,
let $\Gamma_\alpha(\bbP)$ denote the following two players game of length $\alpha$:
At each inning, Players I and II choose conditions of $\bbP$ alternately with
$p_0 \ge q_0 \ge p_1 \ge q_1 \ge \cdots$, but if $\beta<\alpha$ is limit,
at the $\beta$-th inning, Player I does not move and
only Player II chooses a condition $q_\beta$ which is a lower bound of the partial play $\langle p_\xi, q_\zeta \mid \xi,\zeta<\beta$,
$\xi=0$ or successor$\rangle$ (if it is possible):
\begin{center}
\begin{tabular}{c||c|c|c|c|c|c}
 & $0$ & $1$ & $\cdots$       & $\omega$ & $\omega+1$ & $\cdots$ \\
\hline
I & $p_0$ & $p_1$ & $\cdots$ &         & $p_{\omega+1}$ & $\cdots$  \\
\hline
II & $q_0$ & $q_1$ & $\cdots$ & $q_{\omega}$ & $q_{\omega+1}$ &  $\cdots$\\

\end{tabular}
\end{center}
Notice that Player I can always choose $p_{i+1}=q_i$, so the only question is 
whether Player II can choose a condition.
We shall say that Player II \emph{wins} if II can choose a condition at each inning,
and Player I \emph{wins} otherwise.

$\bbP$ is \emph{$\alpha$-strategically closed} if Player II has a winning strategy
in the game $\Gamma_\alpha(\bbP)$.
If $\ka$ is a cardinal and $\bbP$ is $\ka$-strategically closed,
then it is easy to show that $\bbP$ is $\ka$-Baire.
% Then the following hold:

\begin{define}
Let $\ka$ be a regular uncountable cardinal.
$\bbS_\ka$ is the poset consists of all non-reflecting 
bounded subsets of $\ka \cap \Cof(\om)$,
that is, $p \in \bbS_\ka \iff p$ is a bounded subset of $\ka \cap \Cof(\om)$
and $p \cap \alpha$ is non-stationary in $\alpha$ for every $\alpha<\ka$.
Define $p \le q$ if $p$ is an end-extension of $q$.
\end{define}
Clearly $\size{\bbS_\ka}=2^{<\ka}$, hence has the $(2^{<\ka})^+$-c.c.
The following is well-known:
\begin{lemma}
\begin{enumerate}
\item $\bbS_\ka$ is $\ka$-strategically closed.
\item Let $G$ be $(V, \bbS_\ka)$-generic.
Then $\bigcup G$ is a non-reflecting stationary set in $\ka$.
\end{enumerate}
\end{lemma}
\begin{proof}[Sketch of the proof]
(1). For a limit $\beta<\ka$ and a partial play
$\langle p_\xi, q_\zeta \mid \xi,\zeta<\beta$,
$\xi=0$ or successor$\rangle$,
suppose $q=\bigcup_{\zeta<\beta} q_\zeta \in \bbS_\ka$, and let $\gamma=\sup(q)$.
Then Player II takes $q \cup \{\gamma+\om\}$ as his move.
This is a winning strategy of Player II.

(2). To show that $\bigcup G$ is stationary, take a name $\dot C$ for a club in $\ka$.
Take a descending sequence $\seq{p_n \mid n<\om}$ such that
for every $n<\om$,
there is $\alpha_n$ such that $\alpha_n<\sup(p_n)<\alpha_{n+1}$
and $p_{n+1} \Vdash \alpha_n \in \dot C$.
Let $\alpha=\sup_n \alpha_n$, and $p=\bigcup_n p_n \cup \{\alpha\}$.
It is easy to check that $p \in \bbS_\ka$ and
$p \Vdash \alpha \in \dot C \cap \bigcup \dot G$.
\end{proof}
Now we start the proof of Theorem \ref{thm2} (1):
\let\temp\thethm
\renewcommand{\thethm}{\ref{thm2}}
\begin{thm}
If $\ZFC+$``there exists a supercompact cardinal'' is consistent,
then the following theories are consistent as well:
\begin{enumerate}
\item $\ZFC+\Refl(\List)$ holds but $\Refl(\col, \om_2)$ fails.
\item $\ZFC+\Refl(\col)$ holds but $\Refl(\List, \om_2)$ fails.
\end{enumerate}
\end{thm}
\let\thethm\temp
\addtocounter{thm}{-1}

\begin{proof}[Proof of Theorem \ref{thm2} (1)]

Suppose GCH. Let $\ka$ be a supercompact cardinal.
First, force with the poset $\coll(\om_1, <\ka)$.
Let $G$ be $(V, \coll(\om_1, <\ka))$-generic.
We know $\ka=\om_2=2^{\om_1}$ in $V[G]$.
Let $\bbS=\bbS_\ka$, and take a $(V[G], \bbS)$-generic $H$.
We work in $V[G][H]$.
In $V[G][H]$, $\ka=\om_2$, GCH holds, and $S^*=\bigcup H$ is a non-reflecting stationary 
subset of $\om_2 \cap \Cof(\om)$.
Let $\bbP_{\om_3}$ be an $\om_3$-stage countable support iteration of the $\bbP(X,F)$'s which forces:
\begin{enumerate}
\item $\diamondsuit(S)$ holds for every stationary $S \subseteq \om_1$.
\item For every graph $X$ of size $\om_2$,
if $\col(X) \le \om_1$ and $\col(Y) \le \om$ for every $Y \in [X]^{\om_1}$,
then $\List(X) \le \om$.
\item $2^\om=\om_1$ and $2^\la=\la^+$ for every $\la \ge \om_2$.
\end{enumerate}
As in the proof of Theorem \ref{thm1},
$\bbP_{\om_3}$ forces that:
\begin{enumerate}
\item[(4)] For every graph $X$ of size $\om_2$,
if $\List(Y) \le \om$ for every $Y \in [X]^{\om_1}$
then $\List(X) \le \om$.
\end{enumerate}
Since $\bbP_{\om_3}$ satisfies the $\om_2$-c.c.,
it also forces that:
\begin{itemize}
\item[(5)] There is a non-reflecting stationary subset of $\om_2 \cap \Cof(\om)$,
hence $\Refl(\col, \om_2)$ fails.
\end{itemize}
Now we prove that $\bbP_{\om_3}$ also forces that:
\begin{itemize}
\item[(6)] $\FRP(\la)$ holds for every regular $\la>\om_2$.
\end{itemize}
Then we can deduce that $\Refl(\List)$ holds in the generic extension as follows.
We do this by induction on size of graphs.
The $\om_2$ case follows from (4).
Let $X$ be a graph with size $>\om_2$, and suppose every $Y \in [X]^{\om_1}$ has 
countable list-chromatic number.
By the induction hypothesis, we have that every subgraph of $X$ with size $<\size{X}$ has
countable list-chromatic number.
If $\size{X}$ is singular, then we have $\List(X) \le \om$
by (3) and Proposition \ref{1.3}.
If $\size{X}$ is regular,
then we are done by (1), (6) and Proposition \ref{1.7}.

To show (6), consider the following poset $\bbS'$ defined in $V[G][H]$.
$\bbS'$ is the poset consists of all closed bounded subsets $p$ in $\ka$ with
$p \cap S^*=\emptyset$.
Define $p \le q$ if $p$ is an end-extension of $q$.
Clearly $\size{\bbS'}=\om_2$, so $\bbS'$ satisfies the $\om_3$-c.c.
Moreover $\bbS'$ forces that ``$S^*$ is non-stationary''.
The following is straightforward:
\begin{lemma}
In $V[G]$,
let  $D=\{\seq{p, \dot q} \in \bbS *\bbS'\mid 
p \Vdash_\bbS$``$\sup(p)=\max(\dot q)$''$\}$.
Then $D$ is dense in $\bbS*\bbS'$ and
is $\om_2$-closed.
In particular, $\bbS'$ is $\om_2$-Baire in $V[G][H]$.
\end{lemma}

Now take a $(V[G][H], \bbP_{\om_3})$-generic filter $G^*$.
$\bbS *\bbP_{\om_3} * \bbS'$ is forcing equivalent to
$\bbS * \bbS'* \bbP_{\om_3}$.
Since $\bbS * \bbS'$ has a  $\sigma$-closed dense subset,
and $\bbP_{\om_3}$ is $\sigma$-Baire in $V[G]^{\bbS*\bbS'}$,
we know that $\bbS *\bbP_{\om_3} * \bbS'$ is $\sigma$-Baire.
Thus $\bbS'$ remains $\sigma$-Baire in $V[G][H][G^*]$.

Fix a regular cardinal $\la >\om_2$. We also fix a sufficiently large regular cardinal $\theta>\om_3+\la$,
and let $N=\calH_{\theta}^{V[G][H]}$.
To show that $\FRP(\la)$ holds in $V[G][H][G^*]$, take a  stationary $E \subseteq \la \cap \Cof(\om)$
and $g:E \to [\la]^\om$ with $g(\alpha) \in [\alpha]^\om$.
Let $D=\{x \in [\la]^\om\mid \sup(x) \in E, g(\sup(x)) \subseteq x\}$.
$D$ is stationary in $[\la]^\om$.
Take a $(V[G][H][G^*], \bbS')$-generic $H'$, and
let $V^*=V[G][H][G^*][H']$.
$S^*$ is non-stationary in $\om_2$ in $V^*$.
Since $\bbS'$ satisfies the $\om_3$-c.c., we have that
$E$ is stationary in $V^*$.
Moreover, since $\bbS'$ is $\sigma$-Baire, 
it is easy to show that $D$ remains stationary in $V^*$.

In $V$, take a $\theta$-supercompact embedding $j:V \to M$.
We note $j``N \in M$.
Since 
% $V^*=V[G][H][H'][G^*]$ and
$\bbS *\bbS'$ has a $\sigma$-closed dense subset,
we know that $\coll(\om_1, <j(\ka))$ is forcing equivalent
to
$\coll(\om_1, <\ka)*\bbS*\bbS'*\coll(\om_1,[\ka, j(\ka)))$.
Take a $(V^*, \coll(\om_1,[\ka, j(\ka))))$-generic
$G_{tail}$.
Note that $G^*$ is generic over $V[G][H][H'][G_{tail}]$.

In $V^*[G_{tail}]$, we can construct a $(V, \coll(\om_1, <j(\ka)))$-generic $j(G)$
such that $j(G) \cap \coll(\om_1, <\ka)=G$ and $V[G][H][H'][G_{tail}]=V[j(G)]$.
Then $j: V \to M$ can be extended to $j:V[G] \to M[j(G)]$ in $V^*[G_{tail}]$
(actually in $V[G][H][H'][G_{tail}]$). $M[j(G)]$ is still closed under $\om$-sequences in $V^*[G_{tail}]$.

Since $\coll(\om_1, [\ka, j(\ka)))$ is $\sigma$-closed in $V^*$,
$D$ remains stationary in $[\la]^\om$ in $V^*[G_{tail}]$.
We know that $S^* \in M[j(G)]$ and is non-stationary in $\ka$ since $H' \in M[j(G)]$.
So $S^*$ is a non-reflecting bounded subset of $j(\ka)$ in $M[j(G)]$,
and $S^*$ is a condition in $j(\bbS)$.
Take a $(V^*[G_{tail}], j(\bbS))$-generic $j(H)$ with $S^* \in j(H)$.
In $V^*[G_{tail}][j(H)]$, $j$ can be extended to
$j:V[G][H] \to M[j(G)][j(H)]$.
$j(\bbS)$ is $j(\ka)$-strategically closed in $M[j(G)]$ and
$M[j(G)]$ is closed under $\om$-sequences in $V^*[G_{tail}]$.
Hence $j(\bbS)$ is $\om_1$-strategically closed in $V^*[G_{tail}]$,
and $D$ remains  stationary in $V^*[G_{tail}][j(H)]$.
Again, 
because $M[j(G)] \subseteq V[G][H][H'][G_{tail}]$,
we have that $G^*$ is generic over $V[G][H][H'][G_{tail}][j(H)]$.

We know that $j``N \in M \subseteq M[j(G)][j(H)]$,
$\size{j``N}=\om_1$ in $M[j(G)][j(H)]$, and 
$j``N$ is a $\sigma$-closed  elementary submodel of $j(\calH_\theta^{V[G][H]}) \in M[j(G)][j(H)]$.
Hence $j(\bbP_{\om_3}) \cap j``N$ is a complete suborder of $j(\bbP_{\om_3})$.
The following is straightforward:
\begin{claim}
$j \restriction \bbP_{\om_3}$ is a dense embedding
from $\bbP_{\om_3}$ to $j(\bbP_{\om_3}) \cap j``N$.
\end{claim}

Now $M[j(G)][j(H)] \subseteq V[G][H][H'][G_{tail}][j(H)]$
and $G^*$ is generic over \\
$V[G][H][H'][G_{tail}][j(H)]$.
Hence $j$ and $G^*$ induce the filter $G_0$ on $j(\bbP_{\om_3}) \cap j``N$ which is
generic over $V[G][H][H'][G_{tail}][j(H)]$.
In $M[j(G)][j(H)][G_0]$,
we can consider  the quotient $j(\bbP_{\om_3})/G_0$.
Since $j \rest \bbP_{\om_3} \in M[j(G)][j(H)]$,
we know that $G^* \in M[j(G)][j(H)][G_0]$.
Using this observation,
We can check that $j``D \in M[j(G)][j(H)][G_0]$,
$j``D$ is stationary in $[j``\la]^\om$, and  $\size{j``\la}=\om_1$ in $M[j(G)][j(H)][G_0]$.

Take a $(V^*[G_{tail}][j(H)], j(\bbP_{\om_3})/G_0))$-generic $j(G^*)$.
We can canonically extend $j:V[G][H] \to M[j(G)][j(H)]$ to
$j:V[G][H][G^*] \to M[j(G)][j(H)][j(G^*)]$.
$M[j(G)][j(H)][G_0]$ thinks that
$j(\bbP_{\om_3})/G_0$ is $\om_1$-stationary preserving,
thus we have that $j``D$ remains stationary 
in $[j``\la]^\om$ in $M[j(G)][j(H)][j(G^*)]$.
Hence in $M[j(G)][j(H)][j(G^*)]$,
$j``\la$ and $j``D$ witness the statement that 
``there is $I \in [j(\la)]^{\om_1}$ such that
$\sup(I) \notin I$, $\cf(\sup(I))=\om_1$,
and $\{x \in [I]^\om\mid  \sup(x) \in j(E), j(g)(\sup(x)) \subseteq x\}$ 
is stationary in $[I]^\om$''.
By the elementarity of $j$,
it holds in $V[G][H][G^*]$ that 
``there is $I \in [\la]^{\om_1}$ such that
$\sup(I) \notin I$, $\cf(\sup(I))=\om_1$,
and $\{x \in [I]^\om\mid  \sup(x) \in E, g(\sup(x)) \subseteq x\}$ 
is stationary in $[I]^\om$''.
This completes the proof of the condition (6).
\end{proof}

Finally let us pose some questions about $\Refl(\List)$.

\begin{question}
Does $\Refl(\List)$ imply some strong or interesting consequences?
\end{question}
At the moment, we know only Lemma \ref{CH} and Proposition \ref{cons of list}.
The following is a test question:
It is known that the singular cardinal hypothesis follows from $\Refl(\col)$
(Fuchino-Rinot \cite{FR}), hence we would like to ask:
\begin{question}
Does $\Refl(\List)$ imply the singular cardinal hypothesis?
\end{question}
If this question has a positive answer, we can improve the lower bound of the consistency strength
of $\Refl(\List)$.

By Theorem \ref{thm1}, $\Refl(\List, \la)$ does not have a large cardinal strength.
However $2^{\om_1}>\la$ in the resulting model of Theorem \ref{thm1}.
\begin{question}
Does $\Refl(\List, 2^{\om_1})$ have a large cardinal strength?
\end{question}

In the proofs involving $\Refl(\List)$, we always assumed the diamond principle.
\begin{question}
Does $\Refl(\List)$ imply $\diamondsuit(\om_1)$, or $\diamondsuit(S)$ for every stationary $S \subseteq \om_1$?
\end{question}
Concerning this question, Sakai (\cite{Sakai}) told us that the Game Reflection Principle $\mathsf{GRP}$ introduced
by K\"onig \cite{Konig} implies both $\FRP$ and $\diamondsuit(S)$ for every stationary $S \subseteq \om_1$,
hence by Proposition \ref{1.7} it implies $\Refl(\List)$.

\section*{Acknowledgments}

The author would like to greatly thank the referee for many valuable comments and useful suggestions.
This research was supported by JSPS KAKENHI Grant Nos. 18K03403 and 18K03404.

\end{document}